\documentclass[final,1p,times]{elsarticle}

\biboptions{numbers,sort&compress}
\bibliography{mybib}
\usepackage{color}
\usepackage{amssymb}
\usepackage{float}
\usepackage{balance}
 \usepackage{amsthm}
\usepackage{amsmath}
\usepackage{natbib}
\usepackage{epstopdf}
\usepackage[colorlinks,
            linkcolor=red,
            anchorcolor=blue,
            citecolor=green
            ]{hyperref}
\usepackage{multirow}
\usepackage{subfigure}

\usepackage{float}
\usepackage{booktabs}

\usepackage{wasysym}

\def\sgn{\mathop{\rm sgn}}

\def\and{\mathop{\rm and}}
\def\sin{\mathop{\rm sin}}
\def\tanh{\mathop{\rm tanh}}
\makeatletter    
\@addtoreset{equation}{section}
\makeatother    

\newtheorem{assumption}{Assumption}
\newtheorem{corollary}{Corollary}
\newtheorem{lemma}{Lemma}
\newtheorem{definition}{Definition}
\newtheorem{theorem}{Theorem}
\newtheorem{example}{Example}
\newtheorem{remark}{Remark}

\begin{document}

\begin{frontmatter}

\title{Finite-time stability of nonlinear \textcolor{black}{conformable} fractional-order delayed impulsive systems: Impulsive control and perturbation perspectives}

  
\author[a,b]{Lingao Luo}
\ead{lingaoluomail@163.com}
\author[b]{Lulu Li\corref{cor1}}
\ead{lululima@hfut.edu.cn}
\cortext[cor1]{Corresponding author.}
\author[a]{Zhong Liu}
\ead{liuzhong@nudt.edu.cn}
\author[a]{Jianmai Shi}
\ead{jianmaishi@gmail.com}
\address[a]{Laboratory for Big Data and Decision, National University of Defense Technology, Changsha 410073, China}
\address[b]{School of Mathematics, Hefei University of Technology, Hefei 230009, China}

\begin{abstract}
	This paper investigates the finite-time stability (FTS) of nonlinear conformable fractional-order delayed impulsive systems (CFODISs). Using the conformable fractional-order (CFO) derivative framework, we derive a novel FTS result by extending the existing works on continuous integer-order (IO) systems. This result highlights that the settling time of continuous CFO systems depends on the system order and plays a crucial role in discussing FTS scenarios subject to delayed impulses. We establish Lyapunov-based FTS criteria for CFODISs, considering both impulsive control and impulsive perturbation.
	Additionally, we estimate the settling time for both cases, revealing distinct forms compared to the IO case. We apply the theoretical results to delayed impulsive conformable fractional-order memristive neural networks (CFOMNNs) under an elaborately designed controller. 
	We present several simulations to illustrate the validity and applicability of our results.
\end{abstract}


\end{frontmatter}


\section{Introduction}
Fractional-order systems (FOSs) are \textcolor{black}{generalizations} of integer-order systems (IOSs) that can capture the memory and hereditary properties of many complex phenomena \cite{33}. FOSs can model various physical, biological, and engineering systems that exhibit anomalous diffusion, long-range dependence, or fractal behavior \cite{26,501}. The dynamic behavior of FOSs has attracted a lot of attention \textcolor{black}{recently}, especially in terms of stability analysis \cite{30} and synchronization analysis \cite{31}.

Most of the existing works on system stability \textcolor{black}{concentrate} on the asymptotic stability, which means that the system can achieve stability only when time \textcolor{black}{approaches} infinity. However, in many \textcolor{black}{practical} situations, it is desirable to achieve stability within a finite time \textcolor{black}{span}, which is called finite-time stability (FTS) \cite{19,6}.
For example, \cite{19} established the FTS Lyapunov-based framework for continuous IOSs and estimated the settling time. 
\cite{6} \textcolor{black}{analyzed} the finite-time synchronization of integer-order \textcolor{black}{(IO)} neural networks \textcolor{black}{(NNs)} using a special control strategy. 
However, these results cannot be \textcolor{black}{directly} applied to FOSs, \textcolor{black}{because of} the \textcolor{black}{differences} between fraction-order \textcolor{black}{(FO)} calculus and \textcolor{black}{IO} calculus. 
Therefore, some researchers have extended the FTS theory to FOSs and obtained some interesting results \cite{3,12}.
For instance, \cite{3} estimated the settling time for continuous FOSs, which showed a different form from the \textcolor{black}{IO} case.
\cite{12} \textcolor{black}{examined} the finite-time synchronization of \textcolor{black}{memristive} fractional-order \textcolor{black}{NNs} (FONNs). However, most of these works ignored the effects of impulses and impulsive delays on the FTS of FOSs.

In reality, many systems may experience sudden changes at some discrete moments, which are called impulsive effects \cite{351}. Impulses can be classified into two types: destabilizing impulses and stabilizing impulses \cite{352}. The stability problems of impulsive FOSs have been investigated in some works \cite{38,39}. Moreover, from the impulsive control perspective, \cite{16,17,18} discussed the finite-time synchronization of FONNs and estimated the settling time. 
However, a natural question arises: how do impulses affect the FTS and settling time of FOSs? It \textcolor{black}{is worth noting} that the existing works \cite{4,15} addressing this question for impulsive IOSs cannot be extended to impulsive FOSs \textcolor{black}{because of} the nonlinearity and complexity of \textcolor{black}{FO} calculus. Therefore, it is necessary to develop \textcolor{black}{new approaches} and criteria to analyze the FTS of impulsive FOSs under different conditions. 

Furthermore, time delays are \textcolor{black}{unavoidable} in the transmission of impulsive signals, often referred to as impulsive delays \cite{36}. The \textcolor{black}{impacts} of delayed impulses on the asymptotic stability of FONN have been \textcolor{black}{investigated} in \cite{37}. 
To the best of our knowledge, there are no existing results regarding the FTS of nonlinear \textcolor{black}{FO} delayed impulsive systems (FODISs). This is a significant and challenging problem, since impulsive delays can significantly affect the system performance and stability. For example, \cite{41} illustrated through simulations that impulsive delays may destroy the synchronization process of delayed reaction-diffusion FONNs, \textcolor{black}{whereas \cite{219} numerically suggested that impulsive delays within a certain range might enhance the stability of Hopfield switched FONN}. Additionally, \cite{42} rigorously established the positive effect of impulsive delays. These works indicate that impulsive delays can have different effects on system dynamics, depending on their magnitude and distribution.

\textcolor{black}{Note that the aforementioned works on FOSs relied on the Caputo and Riemann-Liouville FO derivatives frameworks, which are limited in certain respects. In particular, these FO derivatives do not adhere to fundamental properties like the product, quotient, and chain rules. Furthermore, in the Riemann-Liouville FO framework, the derivative of an arbitrary scalar does not equal zero. Moreover, the Caputo and Riemann-Liouville FO frameworks do not allow us to determine the monotonic behavior of a function solely based on the sign of its derivative. Furthermore, their complex definition structures introduce challenges when applying them in real-world engineering scenarios. Interestingly, Khalil et al. introduced the conformable FO (CFO) calculus in 2014 \cite{200}, offering a simpler FO concept that overcomes these limitations and establishes a closer relationship to classical IO calculus. Subsequently, conformable FOSs (CFOSs), a widely researched topic, have found extensive applications across  various disciplines, such as 
the grey model \cite{204}, 
fault detection observer design \cite{206}, 
neural networks \cite{209} 
and so on. Moreover, the analysis of the dynamical characteristics in CFOSs has garnered increasing interest, encompassing both infinite-time stability \cite{212} and finite-time stability \cite{214,215}. However, the results obtained in \cite{214,215} are inapplicable in demonstrating the FTS of nonlinear CFOSs. Additionally, these results are restricted to scenarios without impulses. While \cite{216,217} considered impulsive effects in investigating the asymptotic stability problem for delayed CFOSs, they did not account for impulsive delays. As far as the authors know, there are no existing works discussing the FTS of nonlinear conformable FODISs (CFODISs). }

Motivated by the \textcolor{black}{aforementioned} discussions, this paper aims to investigate the FTS of nonlinear CFODISs and makes the following contributions:

\begin{itemize}
	\item [$(1)$] \textcolor{black}{As a crucial initial step, we present a new FTS result for continuous CFOSs within the framework of the CFO derivative proposed in \cite{200}. This result highlights that the settling time of CFOSs depends on both the initial value and the system order. Importantly, our work generalizes and improves upon existing renowned results obtained for IOSs and extends them to FOSs. }
	\item [$(2)$] 
	Unlike the existing works on FTS results for continuous IOSs in \cite{19,6} and continuous Caputo FOSs in \cite{3,12}, we take into account the impulsive effects on the FTS of CFODISs.
	Compared with the previous results \cite{16,17,18,15}, we also consider time delays in impulses and guarantee the FTS property of general nonlinear CFODISs from both impulsive control and impulsive perturbation perspectives.
	
	\item [$(3)$] We focus on CFODISs, in contrast to \cite{4}, where impulsive IOSs were considered. We develop several FTS criteria for general CFODISs, which generalize the results in \cite{4} to the delayed impulsive fractional-order scenario. The results show that when the system \textcolor{black}{undergoes} delayed impulses, stabilizing delayed impulses can reduce the settling time and destabilizing delayed impulses can increase the settling time, compared with the case without impulses.
	
	
\end{itemize}

\textit{Notations:} Let $\mathbb{R}$, $\mathbb{R}^{0}$, $\mathbb{R}^{+}$, $\mathbb{R}^n$, and $\mathbb{Z}^{+}$ denote the sets of real numbers, nonnegative real numbers, positive real numbers, $n$-dimensional real vectors, and positive integers, respectively. 
A function $\varphi:\,\mathbb{R}^0\to\mathbb{R}^0$ is said to belong to the class $\mathcal K$ if $\varphi$ is continuous, strictly increasing, and $\varphi(0)=0$. For any $a_1,\,a_2\in\mathbb{R}$, we use $a_1\wedge a_2$ to denote the minimum of $a_1$ and $a_2$. \textcolor{black}{Let $co[a_1,\,a_2]$ denote the closure of the set $[a_1,\,a_2]$, when $a_1<a_2$. }
\textcolor{black}{For any $a\in\mathbb{Z}^{+}$, $\Omega _a=\{1,\,\cdots,\,a\}$. }$||\cdot||$ is the Euclidean norm.

\section{Preliminaries} \label{sec: Preliminaries}

In this section, we review \textcolor{black}{several fundamental} definitions and lemmas 
and present a general nonlinear CFODIS.
\begin{definition}\label{def-fraction-derivative}
\textcolor{black}{\cite{200,201} The $\mathfrak{Q} $-order CFO derivative starting from $t_{\jmath-1}$ of a function $\mathfrak{H}  (t)\,(t\in [t_{\jmath-1},\,t_{\jmath}),\,\jmath\in\mathbb{Z}^{+})$ is defined as
\begin{equation}\label{def-conformablederivative}
_{t_{\jmath-1}}\mathfrak{T} ^{\mathfrak{Q} }_{t} \mathfrak{H} (t)=\lim_{\varepsilon  \to 0} \frac{\mathfrak{H} \left(t+\varepsilon (t- t_{\jmath-1})^{1-\mathfrak{Q} }\right)-\mathfrak{H} (t)}{\varepsilon }, \qquad t\in (t_{\jmath-1},\,t_{\jmath}), \nonumber
\end{equation}
and
\begin{equation}
	_{t_{\jmath-1}}\mathfrak{T} ^{\mathfrak{Q} }_{t_{\jmath-1}} \mathfrak{H} (t_{\jmath-1}) =\lim_{t\to {t_{\jmath-1}^+}} {_{t_{\jmath-1}}\mathfrak{T} ^{\mathfrak{Q} }_{t}} \mathfrak{H} (t), \qquad t=t_{\jmath-1}, \nonumber
\end{equation}
where $\mathfrak{Q} \in (0,\,1]$. If the limits exist, $\mathfrak{H} (t)$ is said to be  $\mathfrak{Q}$-differentiable. }
\end{definition}

\begin{definition}\label{def-fraction-integral}
\textcolor{black}{	\cite{200,201} The $\mathfrak{Q} $-order CFO integral starting from $t_{\jmath-1}$ of a function $\mathfrak{H}  (t)\,(t\in [t_{\jmath-1},\,t_{\jmath}),\,\jmath\in\mathbb{Z}^{+})$ is defined as
	\begin{equation}\label{def-integral}
		 _{t_{\jmath-1}}\mathfrak{T} ^{-\mathfrak{Q} }_{t} \mathfrak{H} (t)=\int_{t_{\jmath-1}}^t{\mathfrak{H} (s)(s-t_{\jmath-1})^{\mathfrak{Q} -1}ds}, \qquad t \in [t_{\jmath-1},\,t_{\jmath}), \nonumber
	\end{equation}
	where $\mathfrak{Q} \in (0,\,1]$. }
\end{definition}

\begin{lemma}\label{lemma-caputo-micidao-xiandaozaiji-absolute-daoshu}
\textcolor{black}{	\cite{200,201,212} Assume that $\mathfrak{Q} \in(0,1]$ and $\mathfrak{H}_1 (t),\, \mathfrak{H}_2 (t)$ are $\mathfrak{Q}$-differentiable for $t\geq t_0$. Then, it holds that
	\begin{itemize}
		\item [$(1)$] $	_{t_{0}}\mathfrak{T}^{\mathfrak{Q} }_{t} t ^{\nu} = \nu t^{\nu- \mathfrak{Q}} $, for arbitrary $\nu\in\mathbb{R}$;
		\item [$(2)$] ${_{t_{0}}\mathfrak{T}^{\mathfrak{Q} }_{t}} \mathfrak{H} (t)=\left\{{_{t_{0}}\mathfrak{T}^{\mathfrak{Q} }_{t}} \mathfrak{H}_1\right\}(\mathfrak{H}_2(t)) \cdot {_{t_{0}}\mathfrak{T}^{\mathfrak{Q} }_{t}} \mathfrak{H}_2(t)\cdot  \mathfrak{H}_2^{ \mathfrak{Q} -1}(t)$, for $\mathfrak{H} (t)=\mathfrak{H}_1 (\mathfrak{H}_2 (t))$;
		\item [$(3)$] $_{t_0}\mathfrak{T}^{-\mathfrak{Q} }_{t}{_{t_{0}}\mathfrak{T}^{\mathfrak{Q} }_{t}} \mathfrak{H}_1 (t)=\mathfrak{H}_1 (t)-\mathfrak{H}_1 (t_0)$;
		\item [$(4)$] ${_{t_{0}}\mathfrak{T}^{\mathfrak{Q} }_{t}} \left( \nu_1 \mathfrak{H}_1 (t) + \nu_2 \mathfrak{H}_2 (t) \right) = \nu_1 \cdot {_{t_{0}}\mathfrak{T}^{\mathfrak{Q} }_{t}} \mathfrak{H}_1 (t) + \nu_2 \cdot {_{t_{0}}\mathfrak{T}^{\mathfrak{Q} }_{t}} \mathfrak{H}_2 (t)$, for arbitrary $\nu_1,\, \nu_2 \in\mathbb{R}$;
		\item [$(5)$] $\mathfrak{H}_1 (t)$ is monotonically nondecreasing on $t$, if ${_{t_{0}}\mathfrak{T}^{\mathfrak{Q} }_{t}} \mathfrak{H}_1 (t) \geq 0$;
		\item [$(6)$] ${_{t_{0}}\mathfrak{T}^{\mathfrak{Q} }_{t}} \{\mathfrak{H}_1^T(t) \mathfrak{H}_1(t) \}= 2 \mathfrak{H}_1^T(t) \cdot {_{t_{0}}\mathfrak{T}^{\mathfrak{Q} }_{t}} \mathfrak{H}_1(t)$. 
	\end{itemize} }
\end{lemma}

%
%

%


Consider the following general nonlinear CFODIS:
\begin{equation} \label{delay-impulse-system}
\left\{
\begin{array}{lcl}
 _{t_{\jmath-1}}\mathfrak{T} ^{\mathfrak{Q} }_{t} \mathcal{S}  (t)=\mathcal{F} (\mathcal{S}  (t)), & t\geq t_0,\,t\in [t_{\jmath-1},\,t_{\jmath}),\\
 \mathcal{S}  (t_{\jmath})=\mathcal{G} _\jmath(\mathcal{S}  (t_{\jmath}^--\tau_\jmath)), & t= t_{\jmath},\,\jmath\in\mathbb{Z}^{+},\\
 \mathcal{S}  (t_0)=\mathcal{S}  _0,
\end{array}
\right.
\end{equation}
where \textcolor{black}{$\mathfrak{Q} \in(0,1]$} is the \textcolor{black}{order of the system}, $\mathcal{S}  \in\mathbb{R}^n$ is the state vector, and $t_0=0$. The system is subject to impulses at discrete moments $\{t_{\jmath}\}_{\jmath\in\mathbb{Z}^{+}}$, where $t_{\jmath-1}\leq t_\jmath-\tau_\jmath < t_\jmath$. The impulses have delays $\tau_{\jmath}\in\mathbb{R}^+$, which satisfy $\tau_\jmath \leq \tau$ for a constant $\tau\in\mathbb{R}^+$.
$\mathcal{F} :{\mathbb{R}}^{n}\to\mathbb{R}^{n}$ and $\mathcal{G} _\jmath(\mathcal{S} (t)):{\mathbb{R}}^{n}\to\mathbb{R}^{n}$ are continuous vector functions with $\mathcal{F} (0)=0$ and $\mathcal{G} _\jmath(0)=0$. Besides, $\mathcal{S}  (t_\jmath^+)=\mathcal{S}  (t_\jmath)$ and $\mathcal{S}  (t_\jmath^-)=\lim_{t\rightarrow t_\jmath^-}\mathcal{S}  (t)$ exists.

\begin{definition}\label{def-positive-definite-function}
\cite{PDF} A continuous function $\mathcal{V}:{\mathbb{R}}^{n}\to\mathbb{R}^{0}$ is called a positive-definite function (PDF) if $\mathcal{V}(0)=0$ and $\mathcal{V}(\mathcal{S}  )>0$ when $\mathcal{S}  \neq 0$.
\end{definition}

\textcolor{black}{For symbolic simplicity, let $\mathcal{V}(t)$ represent $\mathcal{V}(\mathcal{S}(t))$.}

\begin{definition}\label{def-finite-time-stable}
The system (\ref{delay-impulse-system}) is said to achieve finite-time stability (FTS), if there exists a scalar $T_0\in\mathbb{R}^+$ such that
\begin{equation}
	\lim_{t\rightarrow T_0} \mathcal{S}  (t)=0,\nonumber
\end{equation}
and
\begin{equation}
	\mathcal{S}  (t)\equiv 0, \qquad \forall t\geq T_0,\nonumber
\end{equation}
where $T_0$ \textcolor{black}{is} the settling time.
\end{definition}


\begin{lemma}\label{lemma-finite-time-no-impulse}
	Consider the system (\ref{delay-impulse-system}) without impulses. Assume that the CFO derivative of a PDF $\mathcal{V}(t)$ along the state trajectory of the system (\ref{delay-impulse-system}) satisfies the following inequality:
	\begin{equation}\label{lemma-finite-time-no-impulse1}
		_{t_{0}}\mathfrak{T} ^{\mathfrak{Q} }_{t} \mathcal{V}(t)\leq -c \mathcal{V}^\eta(t),\qquad t\geq t_0,
	\end{equation}
	where $c\in\mathbb{R}^+$ and \textcolor{black}{$\eta\in(0,\,1 )$} are constants. Then, the system (\ref{delay-impulse-system}) without impulses can achieve FTS. Moreover, the following inequality holds:
	\begin{equation}
		\mathcal{V}^{1 -\eta}(t)\leq \mathcal{V}^{1 -\eta} (t_0)-\alpha (t-t_0)^\mathfrak{Q} ,\qquad t_0 \leq t \leq \Gamma_{\mathcal{S}_0},\nonumber
	\end{equation}
	and
	\begin{equation}
		\mathcal{V}(t)\equiv 0,\qquad   \forall t\geq \Gamma_{\mathcal{S}_0},\nonumber
	\end{equation}
	where $\Gamma_{\mathcal{S}_0}=t_0+\left(\frac{\mathcal{V}^{1 -\eta} (t_ 0)}{\alpha}\right)^{\frac{1}{\mathfrak{Q} }}$ and \textcolor{black}{$\alpha=\frac{c (1-\eta )}{\mathfrak{Q}}>0$}.
\end{lemma}
\begin{proof}
\textcolor{black}{	Consider the following differential equation
	\begin{equation}\label{proof-finite-time-no-impulse:2}
		_{t_{0}}\mathfrak{T} ^{\mathfrak{Q} }_{t} \mathcal{W}(t) = -c \mathcal{W}^\eta(t),\qquad t\geq t_0,
	\end{equation}
	with an initial value $\mathcal{W}(t_0) = \mathcal{V}(t_0)$.
%
	Multiplying both sides of (\ref{proof-finite-time-no-impulse:2}) by $(1-\eta ) \mathcal{W}^{-\eta }(t )$ generates that
	\begin{equation}\label{proof-finite-time-no-impulse:1}
	(1-\eta ) \mathcal{W}^{-\eta }(t ) \cdot	{_{t_{0}}\mathfrak{T} ^{\mathfrak{Q} }_{t}} \mathcal{W}(t) = -c (1-\eta ).
	\end{equation}
	Based on Lemma \ref{lemma-caputo-micidao-xiandaozaiji-absolute-daoshu}. (1) and (2), we have
	\begin{align}
		{_{t_{0}}\mathfrak{T} ^{\mathfrak{Q} }_{t}} \mathcal{W}^{1-\eta }(t) &= {_{t_{0}}\mathfrak{T} ^{\mathfrak{Q} }_{t}} t^{1-\eta } |_{t=\mathcal{W}(t)} \cdot {_{t_{0}}\mathfrak{T} ^{\mathfrak{Q} }_{t}} \mathcal{W}(t) \cdot \mathcal{W}^{\mathfrak{Q}-1}(t) \nonumber \\
		&= (1-\eta) \mathcal{W}^{1-\eta - \mathfrak{Q} }(t) \cdot {_{t_{0}}\mathfrak{T} ^{\mathfrak{Q} }_{t}} \mathcal{W}(t) \cdot \mathcal{W}^{ \mathfrak{Q} -1}(t) \nonumber \\
		&= (1-\eta) \mathcal{W}^{-\eta  }(t) \cdot {_{t_{0}}\mathfrak{T} ^{\mathfrak{Q} }_{t}} \mathcal{W}(t), \nonumber
	\end{align}
	which, together with (\ref{proof-finite-time-no-impulse:1}), yields that
	\begin{equation}\label{proof-finite-time-no-impulse:0}
		{_{t_{0}}\mathfrak{T} ^{\mathfrak{Q} }_{t}} \mathcal{W}^{1-\eta }(t)= -c (1-\eta ).
	\end{equation}
    Taking the $\mathfrak{Q} $-order CFO integral on both sides of (\ref{proof-finite-time-no-impulse:0}) from $t_0$ to $t$, and applying Lemma \ref{lemma-caputo-micidao-xiandaozaiji-absolute-daoshu}. (3), we get
	\begin{align}\label{proof-finite-time-no-impulse:3}
		\mathcal{W}^{1-\eta }(t) -\mathcal{W}^{1-\eta }(t_0) &= \int_{t_{0}}^t{[-c (1-\eta )](s-t_{0})^{\mathfrak{Q} -1}ds} \nonumber \\
		& = -\frac{c (1-\eta )}{\mathfrak{Q}} (t-t_0)^\mathfrak{Q} \nonumber \\
		& \triangleq -\alpha (t-t_0)^\mathfrak{Q},
	\end{align}
	where $\alpha=\frac{c (1-\eta )}{\mathfrak{Q}}>0$. Let $\mathcal{H} (t)= \mathcal{W}^{1-\eta }(t_0) - \alpha (t-t_0)^\mathfrak{Q}$. Note that $\alpha >0$ implies the strictly monotonically decreasing property of $\mathcal{H} (t)$. It is evident that $t=\Gamma_{\mathcal{S}_0}$ uniquely satisfies $\mathcal{H} (t) =0$, where $\Gamma_{\mathcal{S}_0}=t_0+\left(\frac{\mathcal{V}^{1 -\eta} (t_ 0)}{\alpha}\right)^{\frac{1}{\mathfrak{Q} }}$. In addition, one has $\mathcal{H} (t) \leq 0$, for any $t\geq \Gamma_{\mathcal{S}_0} $, implying that $\mathcal{W}(t ) \leq 0$ holds for all $t\geq \Gamma_{\mathcal{S}_0}$. On the other hand, according to Lemma \ref{lemma-caputo-micidao-xiandaozaiji-absolute-daoshu}. (4), one has ${_{t_{0}}\mathfrak{T} ^{\mathfrak{Q} }_{t}} ( \mathcal{W}(t ) - \mathcal{V}(t )) \geq 0$. It follows from Lemma \ref{lemma-caputo-micidao-xiandaozaiji-absolute-daoshu}. (5) that $\mathcal{W}(t ) - \mathcal{V}(t ) \geq \mathcal{W}(t_0 ) - \mathcal{V}(t _0) $, which directly deduces that $\mathcal{W}(t ) \geq \mathcal{V}(t )$. Due to the positive definiteness of $\mathcal{V}(t )$, we can conclude that $\mathcal{W}(t ) \geq 0$ for any $t\geq \Gamma_{\mathcal{S}_0}$. Therefore, we deduce that $\mathcal{W}(t ) \equiv  0, \, \forall t\geq \Gamma_{\mathcal{S}_0}$. As a result, we establish that 
	\begin{equation}
		\mathcal{V}^{1 -\eta}(t)\leq \mathcal{W}^{1 -\eta}(t) =\mathcal{V}^{1 -\eta} (t_0)-\alpha (t-t_0)^\mathfrak{Q} ,\qquad t_0 \leq t \leq \Gamma_{\mathcal{S}_0},\nonumber
	\end{equation}
	and simultaneously,
	\begin{equation}
		\mathcal{V}(t)\equiv 0,\qquad   \forall t\geq \Gamma_{\mathcal{S}_0},\nonumber
	\end{equation}
	which completes the proof. }
\end{proof}

\begin{remark}\label{remark-no-impulse}
	\textcolor{black}{When the system order satisfies $\mathfrak{Q}=1$, (\ref{lemma-finite-time-no-impulse1}) simplifies to the classical IO differential inequality in \cite{19,218}. In those works, FTS of continuous IOSs was addressed and the settling time was determined for a given initial state. In the case of order-dependent systems, the settling time naturally depends on the system order. Lemma \ref{lemma-finite-time-no-impulse} builds upon the concepts from \cite{19,218}, transitioning from continuous IOSs to CFOSs, and provides a more adaptable estimation of the settling time in the fractional-order context. }
\end{remark}

\section{Main results}\label{sec: Main results}
In this section, 
we derive the FTS criteria of CFODIS (\ref{delay-impulse-system}) from both impulsive control and impulsive perturbation perspectives. First, we consider the case where CFODIS (\ref{delay-impulse-system}) is subject to stabilizing delayed impulses.
\begin{theorem}\label{theo-finite-stable-impulse}
Suppose that 
there exists a piecewise right-continuous PDF $\mathcal{V}(t) \triangleq \mathcal{V}(\mathcal{S}(t))$ satisfying 
\begin{equation}\label{theo-stable-impulse-V-K-function}
\pi_1(||\mathcal{S}||)\leq \mathcal{V}(t)\leq \pi_2(||\mathcal{S}||),
\end{equation}
\begin{equation} \label{theo-stable-impulse-V-yuanxitiong}
 _{t_{\jmath-1}}\mathfrak{T} ^{\mathfrak{Q}  }_{t} \mathcal{V}(t)\leq -c \mathcal{V}^\eta(t), \qquad t\geq t_0,\,t\in [t_{\jmath-1},\,t_{\jmath}),
\end{equation}
\begin{equation} \label{theo-stable-impulse-V-impulse-instant}
 \mathcal{V}(t_{\jmath})=\mathcal{V}(\mathcal{G}_\jmath(\mathcal{S}(t_{\jmath}^--\tau_\jmath)))\leq \beta_\jmath^{\frac{1}{1 -\eta}}\mathcal{V}( t_{\jmath}^--\tau_\jmath), \qquad t= t_{\jmath},\,\jmath\in\mathbb{Z}^{+},
\end{equation}
where functions $\pi_1,\,\pi_2\in\mathcal K$, \textcolor{black}{scalars} $c>0,
\,\textcolor{black}{\eta\in(0,\,1 )},\,\beta\in(0,\,1)$, and $\beta_\jmath\in(0,\,\beta]$. Then, the system (\ref{delay-impulse-system}) can achieve FTS if there \textcolor{black}{exist} $N$ impulses
on $[0,\,\Gamma_{\mathcal{S}_0})$ and the $N$-th impulsive instant satisfies
\begin{equation}\label{theo-stable-impulse-impulse-instant}
t_N\leq \bigg(\frac{\gamma^{N}(1-\frac{\beta}{\gamma})}{1-\beta}\Gamma_{\mathcal{S}_0}^\mathfrak{Q} -\frac{\beta}{1-\beta}\tau^\mathfrak{Q} \bigg)^{\frac{1}{\mathfrak{Q} }},
\end{equation}
and
\begin{equation}\label{theo-stable-impulse-gamma-condition}
\gamma^{N}(1-\frac{\beta}{\gamma})\Gamma_{\mathcal{S}_0}^\mathfrak{Q} -\beta\tau^\mathfrak{Q} >0,
\end{equation}
where $\gamma\in(\beta,\,1)$. Moreover, the settling time is $T_1=(\gamma^N)^\frac{1}{\mathfrak{Q} }\Gamma_{\mathcal{S}_0}$, where $\Gamma_{\mathcal{S}_0}$ is defined in Lemma \ref{lemma-finite-time-no-impulse}.
\end{theorem}
\begin{proof}
Without loss of generality, suppose that $\mathcal{S}_0\neq 0$. For $t\in [t_{\jmath-1},\,t_{\jmath}),\,\jmath\in\mathbb{Z}^{+}$, 
by multiplying $   (1-\eta )      {\mathcal{V}}^{-\eta}(t)$ on both sides of (\ref{theo-stable-impulse-V-yuanxitiong}), 
together with Lemma \ref{lemma-caputo-micidao-xiandaozaiji-absolute-daoshu}. (1) and (2), one has 
\begin{equation}\label{proof-stable-impulse-Vyijian-micidao}
{_{t_{\jmath-1}}\mathfrak{T} ^{\mathfrak{Q}  }_{t}}  {\mathcal{V}}^{1 -\eta}(t)\le  -   c (1-\eta ) .
\end{equation}
Based on Lemma \ref{lemma-caputo-micidao-xiandaozaiji-absolute-daoshu}. (3), integrating both sides of (\ref{proof-stable-impulse-Vyijian-micidao}) yields that
\begin{align}\label{proof-stable-impulse-V-fangsuo}
{ {\mathcal{V}}}^{1 -\eta}(t)-{ {\mathcal{V}}}^{1 -\eta}( t_{\jmath-1} )
&\le -\frac{c (1-\eta )}{\mathfrak{Q}} (t-t_{\jmath-1})^\mathfrak{Q} \nonumber\\
&\triangleq -\alpha (t-t_{\jmath-1})^\mathfrak{Q} , \qquad t\in [t_{\jmath-1},\,t_{\jmath}),
\end{align}
where $\alpha=\frac{c (1-\eta )}{\mathfrak{Q}} $. 
Based on (\ref{proof-stable-impulse-V-fangsuo}), one has
\begin{equation}
{\mathcal{V}}^{1 -\eta}( t )\leq { {\mathcal{V}}}^{1 -\eta}(t_{0})-\alpha t^\mathfrak{Q} ,\qquad t\in[0,\,(t_1-\tau_1)\wedge \Gamma_{\mathcal{S}_0}).\nonumber
\end{equation}
If $t_1-\tau_1\geq \Gamma_{\mathcal{S}_0}$, there is no impulse before $\Gamma_{\mathcal{S}_0}$. 
Hence, the system (\ref{delay-impulse-system}) has achieved FTS before $t_1-\tau_1$ and the impulses \textcolor{black}{have} no \textcolor{black}{influence} on the FTS of the system.
On the other hand, when $t_1-\tau_1< \Gamma_{\mathcal{S}_0}$, 
there are $N$ impulses on $[0,\,\Gamma_{\mathcal{S}_0})$, i.e., $0<t_1<\cdots<t_N<\Gamma_{\mathcal{S}_0}$. Then, it follows from (\ref{theo-stable-impulse-impulse-instant}) that
\begin{equation}\label{proof-stable-impulse-impulse-instant1}
t_\jmath\le t_N \le \bigg(\frac{\gamma^{N}(1-\frac{\beta}{\gamma})}{1-\beta}\Gamma_{\mathcal{S}_0}^\mathfrak{Q} -\frac{\beta}{1-\beta}\tau^\mathfrak{Q}  \bigg)^{\frac{1}{\mathfrak{Q} }}\le \bigg(\frac{\gamma^{\jmath}(1-\frac{\beta}{\gamma})}{1-\beta}\Gamma_{\mathcal{S}_0}^\mathfrak{Q} -\frac{\beta}{1-\beta}\tau^\mathfrak{Q}  \bigg)^{\frac{1}{\mathfrak{Q} }},\, \jmath \in \Omega _N, 
\end{equation}
i.e.,
\begin{equation}\label{proof-stable-impulse-impulse-instant2}
\beta \gamma ^{\jmath-1}+\frac{t_\jmath^\mathfrak{Q} }{\Gamma _{\mathcal{S}_0}^\mathfrak{Q} } (1-\beta )+\frac{\beta \tau ^\mathfrak{Q} }{\Gamma _{\mathcal{S}_0}^\mathfrak{Q} }\le \gamma ^\jmath,\, \jmath \in \Omega _N, 
\end{equation}
By (\ref{proof-stable-impulse-V-fangsuo}), one has
\begin{equation}\label{proof-stable-impulse-V-fangsuo-0-1}
{\mathcal{V}}^{1 -\eta}( t )\leq { {\mathcal{V}}}^{1 -\eta}(t_{0})-\alpha t^\mathfrak{Q} ,\qquad t\in[0,\,t_{1}).
\end{equation}
When $t\in [t_1,\,t_2)$, it follows from (\ref{theo-stable-impulse-V-impulse-instant}), (\ref{proof-stable-impulse-V-fangsuo}),  (\ref{proof-stable-impulse-impulse-instant2}), and (\ref{proof-stable-impulse-V-fangsuo-0-1})
that
\begin{align}\label{proof-stable-impulse-V-fangsuo-1-2}
{\mathcal{V}}^{1 -\eta}( t )&\leq \mathcal{V}^{1 -\eta}( t_{1} )-\alpha (t-t_{1})^\mathfrak{Q} \nonumber\\
&\le \beta _1 \mathcal{V}^{1 -\eta}( t_{1}^--\tau _1 )-\alpha (t-t_{1})^\mathfrak{Q} \nonumber \\
&\le \beta [\mathcal{V}^{1 -\eta}(t_0)-\alpha (t_{1}^--\tau _1)^\mathfrak{Q} ]-\alpha (t-t_{1})^\mathfrak{Q} \nonumber \\
&\le \beta \mathcal{V}^{1 -\eta}(t_0)-\alpha\beta  (t_{1}^\mathfrak{Q} -\tau _1^\mathfrak{Q} )-\alpha (t^\mathfrak{Q} -t_{1}^\mathfrak{Q} )\nonumber \\
&\le \beta \mathcal{V}^{1 -\eta}(t_0)+\alpha \beta \tau^\mathfrak{Q} +\alpha (1-\beta )t_{1}^\mathfrak{Q} -\alpha t^\mathfrak{Q} \nonumber \\
&= \left[\beta +\frac{t_{1}^\mathfrak{Q} }{\Gamma _{\mathcal{S}_0}^\mathfrak{Q} }(1-\beta )+\frac{\beta \tau^\mathfrak{Q} }{\Gamma _{\mathcal{S}_0}^\mathfrak{Q} } \right]\mathcal{V}^{1 -\eta}(t_0) -\alpha t^\mathfrak{Q} \nonumber \\
&\le \gamma {\mathcal{V}}^{1 -\eta}(t_0)-\alpha t^\mathfrak{Q}. 
\end{align}
Similarly, for $t\in [t_\jmath,\,t_{\jmath+1}),\,\textcolor{black}{\jmath \in \Omega _N}$,  
one has
\begin{align}\label{proof-stable-impulse-V-fangsuo-j-(j+1)}
{\mathcal{V}}^{1 -\eta}( t )&\leq \mathcal{V}^{1 -\eta}( t_{\jmath} )-\alpha (t-t_{\jmath})^\mathfrak{Q} \nonumber\\
&\le \beta _\jmath \mathcal{V}^{1 -\eta}( t_{\jmath}^--\tau _\jmath )-\alpha (t-t_{\jmath})^\mathfrak{Q} \nonumber \\
&\le \beta [\gamma ^{\jmath-1} \mathcal{V}^{1 -\eta}(t_0)-\alpha (t_{\jmath}^--\tau _\jmath)^\mathfrak{Q} ]-\alpha (t-t_{\jmath})^\mathfrak{Q} \nonumber \\
&\le \beta\gamma ^{\jmath-1} \mathcal{V}^{1 -\eta}(t_0)-\alpha\beta  (t_{\jmath}^\mathfrak{Q} -\tau _\jmath^\mathfrak{Q} )-\alpha (t^\mathfrak{Q} -t_{\jmath}^\mathfrak{Q} )\nonumber \\
&\le \beta\gamma ^{\jmath-1} \mathcal{V}^{1 -\eta}(t_0)+\alpha \beta \tau^\mathfrak{Q} +\alpha (1-\beta )t_{\jmath}^\mathfrak{Q} -\alpha t^\mathfrak{Q} \nonumber \\
&= \left[\beta\gamma ^{\jmath-1} +\frac{t_{\jmath}^\mathfrak{Q} }{\Gamma _{\mathcal{S}_0}^\mathfrak{Q} }(1-\beta )+\frac{\beta \tau^\mathfrak{Q} }{\Gamma _{\mathcal{S}_0}^\mathfrak{Q} } \right]\mathcal{V}^{1 -\eta}(t_0) -\alpha t^\mathfrak{Q} \nonumber \\
&\le \gamma ^{\jmath} {\mathcal{V}}^{1 -\eta}(t_0)-\alpha t^\mathfrak{Q} ,
\end{align}
where $t_{N+1}\triangleq (\gamma^N)^\frac{1}{\mathfrak{Q} }\Gamma_{\mathcal{S}_0}$. 
Furthermore, it follows from (\ref{proof-stable-impulse-impulse-instant1}) that 
\begin{align}\label{proof-stable-impulse-impulse-instant3}
t_{\jmath+1}&\le \bigg(\frac{\gamma^{\jmath+1}(1-\frac{\beta}{\gamma})}{1-\beta}\Gamma_{\mathcal{S}_0}^\mathfrak{Q} -\frac{\beta}{1-\beta}\tau^\mathfrak{Q}  \bigg)^{\frac{1}{\mathfrak{Q} }}\nonumber\\
&\le \bigg(\frac{\gamma^{\jmath+1}(1-\frac{\beta}{\gamma})}{1-\beta}\bigg)^{\frac{1}{\mathfrak{Q} } }\Gamma_{\mathcal{S}_0}-\bigg(\frac{\beta}{1-\beta}\bigg)^{\frac{1}{\mathfrak{Q} } }\tau\nonumber\\
&< \bigg(\frac{\gamma^{\jmath}(\gamma -\beta)}{1-\beta}\bigg)^{\frac{1}{\mathfrak{Q} } }\Gamma_{\mathcal{S}_0}\nonumber\\
&< (\gamma ^\jmath)^{\frac{1}{\mathfrak{Q} } }\Gamma_{\mathcal{S}_0}, \qquad \jmath\le N-1.
\end{align}
Note that the right side of (\ref{proof-stable-impulse-V-fangsuo-j-(j+1)}) is monotonically decreasing \textcolor{black}{with respect to} $t$, and  
equals $0$ at $t=t_{N+1}$.
Hence, together with (\ref{proof-stable-impulse-V-fangsuo-j-(j+1)}) and (\ref{proof-stable-impulse-impulse-instant3}), one
derives that $\lim_{t\rightarrow T_1}\mathcal{V}( t )=0$, $\mathcal{V}( t )\leq \mathcal{V}(t_0)$ for $t\ge 0 $, and $\mathcal{V}( t )\equiv 0$ for $t\geq T_1$, where $T_1=(\gamma^N)^\frac{1}{\mathfrak{Q} }\Gamma_{\mathcal{S}_0}$. Then, 
one concludes that for $t\geq T_1$, 
$\mathcal{S}(t)\equiv 0$ 
and $\lim_{t\rightarrow T_1}\mathcal{S}(t)= 0$. Thus, the system (\ref{delay-impulse-system}) can achieve FTS with the settling time $T_1$.
\end{proof}


In particular, if impulses are delay-independent, i.e., $\tau_\jmath=0,\,\jmath\in \mathbb{Z}^+$, then \textcolor{black}{we obtain the following corollary}.
\begin{corollary}\label{coro-theo-stable-impulse}
Consider the system (\ref{delay-impulse-system}) with  $\tau_\jmath=0,\,\jmath\in \mathbb{Z}^+$. Assume that there exist functions $\pi_1,\,\pi_2\in\mathcal K$, \textcolor{black}{scalars} $c>0,\,\textcolor{black}{\eta\in(0,\,1 )},\,\beta\in(0,\,1),\,\beta_\jmath\in(0,\,\beta]$, 
such that the piecewise right-continuous PDF $\mathcal{V}( t )$ satisfies (\ref{theo-stable-impulse-V-K-function}), (\ref{theo-stable-impulse-V-yuanxitiong}), and
\begin{equation} \label{coro-theo-stable-impulse-V-impulse-instant}
 \mathcal{V}( t_{\jmath} )=\mathcal{V}(\mathcal{G}_\jmath(\mathcal{S}(t_{\jmath}^-))\leq \beta_\jmath^{\frac{1}{1 -\eta}}\mathcal{V}( t_{\jmath}^- ), \qquad t= t_{\jmath},\,\jmath\in\mathbb{Z}^{+}.
\end{equation}
Then, 
the system (\ref{delay-impulse-system}) can achieve FTS. Moreover, the settling time is $\tilde{T}_1=({\tilde{\gamma}}^N)^\frac{1}{\mathfrak{Q} }\Gamma_{\mathcal{S}_0}$, if the impulsive sequence $\{t_{\jmath}\}_{\jmath\in\{1,\,\cdots,\,N\}}$ satisfies
\begin{equation}\label{coro-theo-stable-impulse-impulse-instant}
	t_N\leq \bigg(\frac{\tilde{\gamma}^{N}(1-\frac{\beta}{\tilde{\gamma}})}{1-\beta}\bigg)^{\frac{1}{\mathfrak{Q} }}\Gamma_{\mathcal{S}_0},
\end{equation}
where $\tilde{\gamma}\in(\beta,\,1)$ and $\Gamma_{\mathcal{S}_0}$ is defined in Lemma \ref{lemma-finite-time-no-impulse}.
\end{corollary}
\begin{proof}
Without loss of generality, suppose that $\mathcal{S} _0\neq 0$. We then show that the impulsive system (\ref{delay-impulse-system}) is FTS under delay-independent impulses. 
It follows from (\ref{theo-stable-impulse-V-yuanxitiong}) that
\begin{equation}
{\mathcal{V}}^{1 -\eta}( t )\leq { {\mathcal{V}}}^{1 -\eta}(t_{0})-\alpha t^\mathfrak{Q} ,\qquad t\in[0,\,t_1\wedge \Gamma_{\mathcal{S}_0}), \nonumber
\end{equation}
where $\alpha=\frac{c (1-\eta )}{\mathfrak{Q}}$. If $t_1\geq \Gamma_{\mathcal{S}_0}$, it implies that there is no impulse before $\Gamma_{\mathcal{S}_0}$ for the system (\ref{delay-impulse-system}). Then, one concludes that for $t\in[0,\,\Gamma_{\mathcal{S}_0}]$, $\mathcal{V}( t )\leq \mathcal{V}(t_0)$ and for $t\in[\Gamma_{\mathcal{S}_0},\,+\infty)$, $\mathcal{V}( t )\equiv 0$. Otherwise, if $t_1< \Gamma_{\mathcal{S}_0}$, assume that there exists \textcolor{black}{$M\in\mathbb{Z}^+$} such that impulsive instants on $[0,\,\Gamma_{\mathcal{S}_0}]$ satisfy $0<t_1<\cdots<t_M<\Gamma_{\mathcal{S}_0}$. Furthermore, one can derive that
\begin{equation}
{\mathcal{V}}^{1 -\eta}( t )\leq { {\mathcal{V}}}^{1 -\eta}(t_{0})-\alpha t^\mathfrak{Q} ,\qquad t\in[0,\,t_1).\nonumber
\end{equation}
It follows from (\ref{proof-stable-impulse-V-fangsuo}), (\ref{coro-theo-stable-impulse-V-impulse-instant}), and $\beta_1<1$, that
\begin{align}
{\mathcal{V}}^{1 -\eta}( t )&\leq \mathcal{V}^{1 -\eta}( t_{1} )-\alpha (t-t_{1})^\mathfrak{Q} \nonumber\\
&\le \beta _1 \mathcal{V}^{1 -\eta}( t_{1}^- )-\alpha (t-t_{1})^\mathfrak{Q} \nonumber \\
&\le \mathcal{V}^{1 -\eta}(t_0)-\alpha t_{1}^\mathfrak{Q} -\alpha (t^\mathfrak{Q} -t_{1}^\mathfrak{Q} )\nonumber \\
&= \mathcal{V}^{1 -\eta}(t_0)-\alpha t^\mathfrak{Q} , \qquad t\in [t_1,\,t_2).\nonumber
\end{align}
Similarly, for $t\in [t_{\jmath },\,t_{\jmath +1}),\,\textcolor{black}{\jmath \in \Omega _M}$, 
one has
\begin{align}\label{proof-coro-stable-impulse-V-fangsuo-j-(j+1)}
{\mathcal{V}}^{1 -\eta}( t )&\leq \mathcal{V}^{1 -\eta}( t_{\jmath } )-\alpha (t-t_{\jmath })^\mathfrak{Q} \nonumber\\
&\le \beta _\jmath  \mathcal{V}^{1 -\eta}( t_{\jmath }^- )-\alpha (t-t_{\jmath })^\mathfrak{Q} \nonumber \\
&\le \mathcal{V}^{1 -\eta}(t_0)-\alpha t_{\jmath }^\mathfrak{Q} -\alpha (t^\mathfrak{Q} -t_{\jmath }^\mathfrak{Q} )\nonumber \\
&= \mathcal{V}^{1 -\eta}(t_0)-\alpha t^\mathfrak{Q} , 
\end{align}
where $t_{M+1}\triangleq \Gamma_{\mathcal{S}_0}$ such that the right side of (\ref{proof-coro-stable-impulse-V-fangsuo-j-(j+1)}) is exactly zero. Then, one concludes that 
$\lim_{t\rightarrow \Gamma_{\mathcal{S}_0}}\mathcal{V}( t )=0$ and $\mathcal{V} t )\equiv 0$ for $t\geq \Gamma_{\mathcal{S}_0}$. Next, 
one has 
$\mathcal{S}(t)\equiv 0$ for $t\geq \Gamma_{\mathcal{S}_0}$, and $\lim_{t\rightarrow \Gamma_{\mathcal{S}_0}}\mathcal{S}(t)= 0$. Hence, the system (\ref{delay-impulse-system}) under any delay-independent impulses can achieve FTS.

Moreover, when the system (\ref{delay-impulse-system}) is under a certain impulsive sequence $\{t_{\jmath }\}_{\jmath \in\{1,\,\cdots,\,N\}}$ satisfying
(\ref{coro-theo-stable-impulse-impulse-instant}), the \textcolor{black}{rigorous} proof is \textcolor{black}{similar} to the proof of Theorem \ref{theo-finite-stable-impulse}, which is \textcolor{black}{omitted} here due to limited space.
\end{proof}

\begin{remark}\label{remark-theo-stable-impulse}
	\textcolor{black}{Compared with} the \textcolor{black}{scenario} without stabilizing impulses in Lemma \ref{lemma-finite-time-no-impulse}, 
	\textcolor{black}{the settling time} is smaller when the 
	stabilizing delayed impulses are involved in Theorem \ref{theo-finite-stable-impulse}. 
	This means that stabilizing delayed impulses can \textcolor{black}{accelerate} the convergence of the system to equilibrium. Additionally, Corollary \ref{coro-theo-stable-impulse} indicates that the FTS property of the system (\ref{delay-impulse-system}) can always be ensured even if the stabilizing delay-independent impulsive sequence is uncertain and shows that the system (\ref{delay-impulse-system}) has a lower bound of settling time under certain impulsive effects.
\end{remark}


Next, we consider the case that CFODIS (\ref{delay-impulse-system}) suffers from destabilizing delayed impulses.
\begin{theorem}\label{theo-finite-destable-impulse}
Suppose that there exist functions $\pi_1,\,\pi_2\in\mathcal K$, \textcolor{black}{scalars} $c>0,\,\textcolor{black}{\eta\in(0,\,1 )},\,\beta\in[1,\,+\infty ),\,\beta_\jmath \in(0,\,\beta]$, $\gamma\in[\beta,\,+\infty )$,  
such that the piecewise right-continuous PDF $\mathcal{V}( t )$ satisfies (\ref{theo-stable-impulse-V-K-function}), (\ref{theo-stable-impulse-V-yuanxitiong}), and (\ref{theo-stable-impulse-V-impulse-instant}).
Then, the system (\ref{delay-impulse-system}) can achieve FTS if the impulsive sequence $\{t_{\jmath }\}_{\jmath \in\mathbb{Z}^{+}}$ satisfies
\begin{equation}\label{theo-destable-impulse-impulse-instant}
N_0 \triangleq \min \{\jmath \in \mathbb{Z}^{+}:\,t_{\jmath }\ge (\gamma ^{\jmath -1})^{\frac{1}{\mathfrak{Q} }}\Gamma _{\mathcal{S}_0} \}<+\infty ,
\end{equation}
and
\begin{equation}\label{theo-destable-impulse-gamma-condition}
\gamma \ge \beta +\frac{\beta \tau ^\mathfrak{Q} }{\Gamma _{\mathcal{S}_0}^\mathfrak{Q} } .
\end{equation}
Moreover, the settling time is $T_2=(\gamma^{N_0-1})^\frac{1}{\mathfrak{Q} }\Gamma_{\mathcal{S}_0}$, 
where $\Gamma_{\mathcal{S}_0}$ is defined in Lemma \ref{lemma-finite-time-no-impulse}.
\end{theorem}
\begin{proof}
Without loss of generality, suppose that $\mathcal{S}_0\neq 0$. 
From (\ref{proof-stable-impulse-V-fangsuo}), one has
\begin{equation}
{\mathcal{V}}^{1 -\eta}( t )\leq { {\mathcal{V}}}^{1 -\eta}(t_{0})-\alpha t^\mathfrak{Q} ,\qquad t\in[0,\,(t_1-\tau_1)\wedge \Gamma_{\mathcal{S}_0}), \nonumber
\end{equation}
where $\alpha=\frac{c (1-\eta )}{\mathfrak{Q}}$. If $t_1-\tau_1\geq \Gamma_{\mathcal{S}_0}$, it implies that $N_0=1$ and the system (\ref{delay-impulse-system}) has no delay-dependent impulses. In such case, one derives that for $t\in[0,\,\Gamma_{\mathcal{S}_0}]$, $\mathcal{V}( t )\leq \mathcal{V}(t_0)$ holds and for $t\in[\Gamma_{\mathcal{S}_0},\,+\infty)$, $\mathcal{V}( t )\equiv 0$ holds. Conversely, if $t_1-\tau_1< \Gamma_{\mathcal{S}_0}$, 
it follows from (\ref{theo-destable-impulse-gamma-condition}) and $\gamma \ge 1$ that
\begin{align}\label{proof-destable-impulse-gamma-condition}
\beta\gamma ^{\jmath -2}\mathcal{V}^{1 -\eta }(t_0)+\alpha \beta \tau ^\mathfrak{Q} \le \gamma ^{\jmath -1}\mathcal{V}^{1 -\eta }(t_0),\qquad \jmath \ge 2.
\end{align}
Next, for $t\in [t_1,\,t_2\wedge \gamma ^\frac{1}{\mathfrak{Q} }\Gamma _{\mathcal{S}_0})$, 
one has
\begin{align}\label{proof-destable-impulse-V-fangsuo-1-2}
{\mathcal{V}}^{1 -\eta}( t )&\leq \mathcal{V}^{1 -\eta}( t_{1} )-\alpha (t-t_{1})^\mathfrak{Q} \nonumber\\
&\le \beta _1 \mathcal{V}^{1 -\eta}( t_{1}^--\tau _1 )-\alpha (t-t_{1})^\mathfrak{Q} \nonumber \\
&\le \beta [\mathcal{V}^{1 -\eta}(t_0)-\alpha (t_{1}^--\tau _1)^\mathfrak{Q} ]-\alpha (t-t_{1})^\mathfrak{Q} \nonumber \\
&\le \beta \mathcal{V}^{1 -\eta}(t_0)-\alpha\beta  (t_{1}^\mathfrak{Q} -\tau _1^\mathfrak{Q} )-\alpha (t^\mathfrak{Q} -t_{1}^\mathfrak{Q} )\nonumber \\
&\le \beta \mathcal{V}^{1 -\eta}(t_0)+\alpha \beta \tau^\mathfrak{Q} -\alpha t^\mathfrak{Q} \nonumber \\
&\le \gamma \mathcal{V}^{1 -\eta}(t_0)-\alpha t^\mathfrak{Q} .
\end{align}
Similarly, for $t\in [t_{N_0-1},\,t_{N_0}\wedge (\gamma^{N_0-1}) ^\frac{1}{\mathfrak{Q} }\Gamma _{\mathcal{S}_0})$, one derives that
\begin{align}\label{proof-destable-impulse-V-fangsuo-(N0-1)-N0}
{\mathcal{V}}^{1 -\eta}( t )&\leq \mathcal{V}^{1 -\eta}( t_{N_0-1} )-\alpha (t-t_{N_0-1})^\mathfrak{Q} \nonumber\\
&\le \beta _{N_0-1} \mathcal{V}^{1 -\eta}( t_{N_0-1}^--\tau _{N_0-1} )-\alpha (t-t_{N_0-1})^\mathfrak{Q} \nonumber \\
&\le \beta [\gamma^{N_0-2} \mathcal{V}^{1 -\eta}(t_0)-\alpha (t_{N_0-1}^--\tau _{N_0-1})^\mathfrak{Q} ]-\alpha (t-t_{N_0-1})^\mathfrak{Q} \nonumber \\
&\le \beta\gamma^{N_0-2}  \mathcal{V}^{1 -\eta}(t_0)-\alpha\beta  (t_{N_0-1}^\mathfrak{Q} -\tau _{N_0-1}^\mathfrak{Q} )-\alpha (t^\mathfrak{Q} -t_{N_0-1}^\mathfrak{Q} )\nonumber \\
&\le \beta\gamma^{N_0-2}  \mathcal{V}^{1 -\eta}(t_0)+\alpha \beta \tau^\mathfrak{Q} -\alpha t^\mathfrak{Q} \nonumber \\
&\le \gamma^{N_0-1} \mathcal{V}^{1 -\eta}(t_0)-\alpha t^\mathfrak{Q} .
\end{align}
According to (\ref{theo-destable-impulse-impulse-instant}), one has $t_{\jmath }< (\gamma ^{\jmath -1})^{\frac{1}{\mathfrak{Q} }}\Gamma _{\mathcal{S}_0} $ when $\textcolor{black}{\jmath \in \Omega _{N_0-1}}$, and $t_{N_0}\ge (\gamma ^{N_0-1})^{\frac{1}{\mathfrak{Q} }}\Gamma _{\mathcal{S}_0}$. 
Then, one concludes that 
$\lim_{t\rightarrow T_2}\mathcal{V}( t )=0$ and $\mathcal{V}( t )\equiv 0$ for $t\geq T_2$, where $T_2=(\gamma^{N_0-1})^\frac{1}{\mathfrak{Q} }\Gamma_{\mathcal{S}_0}$. Then, 
one has $\mathcal{S}(t)\equiv 0$ for $t\geq T_2$, and $\lim_{t\rightarrow T_2}\mathcal{S}(t)= 0$. Thus, the system (\ref{delay-impulse-system}) can achieve FTS with the settling time $T_2$.
\end{proof}

If impulses are delay-independent, i.e., $\tau_\jmath =0,\,\jmath \in \mathbb{Z}^+$, then \textcolor{black}{we obtain the following corollary}. 

\begin{corollary}\label{coro-theo-destable-impulse}
Consider the system (\ref{delay-impulse-system}) with  $\tau_\jmath =0,\,\jmath \in \mathbb{Z}^+$. Assume that there exist functions $\pi_1,\,\pi_2\in\mathcal K$, \textcolor{black}{scalars} $c>0,\,\textcolor{black}{\eta\in(0,\,1 )},\,\beta\in[1,\,+\infty ),\,\beta_\jmath \in(0,\,\beta]$, 
such that the piecewise right-continuous PDF $\mathcal{V}( t )$ satisfies (\ref{theo-stable-impulse-V-K-function}), (\ref{theo-stable-impulse-V-yuanxitiong}) and (\ref{coro-theo-stable-impulse-V-impulse-instant}). Then, the system (\ref{delay-impulse-system}) can achieve FTS if the impulsive sequence $\{t_{\jmath }\}_{\jmath \in\mathbb{Z}^{+}}$ satisfies
\begin{equation}\label{coro-destable-impulse-impulse-instant}
\tilde{N}_0 \triangleq \min \{\jmath \in \mathbb{Z}^{+}:\,t_{\jmath }\ge (\beta ^{\jmath -1})^{\frac{1}{\mathfrak{Q} }}\Gamma _{\mathcal{S}_0} \}<+\infty .
\end{equation}
Moreover, the settling time is $\tilde{T}_2=(\beta^{\tilde{N}_0-1})^\frac{1}{\mathfrak{Q} }\Gamma_{\mathcal{S}_0}$, where 
$\Gamma_{\mathcal{S}_0}$ is defined in Lemma \ref{lemma-finite-time-no-impulse}.
\end{corollary}

\begin{remark}\label{remark-coro-destable-impulse}
From Theorem \ref{theo-finite-destable-impulse} and Corollary \ref{coro-theo-destable-impulse}, \textcolor{black}{we conclude} that the presence of destabilizing impulses increases the settling time, 
compared with the case without destabilizing impulses.
\end{remark}

\begin{remark}\label{remark-coro-stable-destable-impulse-integer-order}
When $\mathfrak{Q} =1$, the system (\ref{delay-impulse-system}) is reduced to the \textcolor{black}{IO} case. In such situation, if $\beta_\jmath =\beta\in(0,\,1),\,\jmath \in\mathbb{Z}^+$, Corollary \ref{coro-theo-stable-impulse} is simplified to the results 
given in \cite[Theorem 1]{4}. Meanwhile, if $\beta_\jmath =\beta\in[1,\,+\infty ),\,\jmath \in\mathbb{Z}^+$, Corollary \ref{coro-theo-destable-impulse} is equivalent to the results 
proposed in \cite[Theorem 2]{4}. Hence, 
our results extend the FTS of \textcolor{black}{IO} impulsive systems established in \cite{4} to the fractional-order case with delayed impulses.
\end{remark}

\section{Applications}\label{sec: Applications}
In this section, we \textcolor{black}{apply} Theorems \ref{theo-finite-stable-impulse} and \ref{theo-finite-destable-impulse} to analyze the finite-time synchronization problem of \textcolor{black}{CFO} memristive \textcolor{black}{NNs} (CFOMNNs) with delayed impulses. The model we consider is given by 
\begin{equation} \label{delay-impulse-memristive-drive-system}
	\left\{
	\begin{array}{lcl}
		_{t_{\jmath -1}}\mathfrak{T} ^{\mathfrak{Q} }_{t} x_r(t)=-a_{r} x_r(t)+\sum_{s=1}^{n} b_{rs} (x_r(t))f_s(x_s(t))+I_r,  & t\geq t_0,\,t\in [t_{\jmath -1},\,t_{\jmath }),\\
		x_r(t_{\jmath })=\mu _\jmath  x_r(t_{\jmath }^--\tau_\jmath ), & t= t_{\jmath },\,\jmath \in\mathbb{Z}^{+},\\
		x_r(t_0)=x_{r0}, & r\in  \textcolor{black}{\Omega _n},  
	\end{array}
	\right.
\end{equation}
and
\begin{equation} \label{delay-impulse-memristive-response-system}
	\left\{
	\begin{array}{lcl}
		_{t_{\jmath -1}}\mathfrak{T} ^{\mathfrak{Q} }_{t} y_r(t)=-a_{r} y_r(t)+\sum_{s=1}^{n} b_{rs} (y_r(t))f_s(y_s(t))+I_r+u_r(t),  & t\geq t_0,\,t\in [t_{\jmath -1},\,t_{\jmath }),\\
		y_r(t_{\jmath })=\mu _\jmath  y_r(t_{\jmath }^--\tau_\jmath ), & t= t_{\jmath },\,\jmath \in\mathbb{Z}^{+},\\
		y_r(t_0)=y_{r0}, & r\in \textcolor{black}{\Omega _n},   
	\end{array}
	\right.
\end{equation}
where \textcolor{black}{$\mathfrak{Q} \in (0,\,1]$}, $t_0=0$, 
\textcolor{black}{$n$ denotes the amount of units in the NNs}, $x_r(t),\,y_r(t)\in \mathbb{R}$ are respectively the $r$-th neuron states of \textcolor{black}{drive-response systems (\ref{delay-impulse-memristive-drive-system}) and (\ref{delay-impulse-memristive-response-system})}, $a_r\in \mathbb{R}^+$ is the self-feedback connection weight, $f_r(\cdot): {\mathbb{R}}\to\mathbb{R}$ denotes the activation function, $I_r\in \mathbb{R}$ represents the external input, $u_r(t)$ is the controller which will be designed later,  $\mu _\jmath  \in \mathbb{R}$ represents the impulsive intensity at the impulsive instant $t_\jmath $, $\tau_{\jmath }\in\mathbb{R}^+$ is the impulsive delay satisfying $\tau_\jmath  \leq \tau$ for a constant $\tau\in\mathbb{R}^+$, the impulsive sequence is $\{t_{\jmath }\}_{\jmath \in\mathbb{Z}^{+}}$ satisfying $t_{\jmath -1}\leq t_\jmath -\tau_\jmath  < t_\jmath $, $x_{r0},\,y_{r0}$ are initial states of the $r$-th neuron, and $b_{rs} (x_r(t)),\,b_{rs} (y_r(t))$ are respectively memristive synaptic connection weights of \textcolor{black}{drive-response systems (\ref{delay-impulse-memristive-drive-system}) and (\ref{delay-impulse-memristive-response-system})} which are described as
\begin{flalign*}
	b_{rs} (x_r(t))=\left\{
	\begin{aligned}
		\overline{b}_{rs},    \qquad   & \left | x_r(t) \right | \le  \Theta _r\\
		\underline{b}_{rs}, \qquad  & \left | x_r(t) \right | >  \Theta _r\\
	\end{aligned}
	\right.
	,~\and~	b_{rs} (y_r(t))=\left\{
	\begin{aligned}
		\overline{b}_{rs},    \qquad   & \left | y_r(t) \right | \le  \Theta _r\\
		\underline{b}_{rs}, \qquad  & \left | y_r(t) \right | >  \Theta _r\\
	\end{aligned}
	\right.	
	.
\end{flalign*}
Here, $\Theta _r\in\mathbb{R}^+$ characterizes the switching jumps, and constants $\overline{b}_{rs},\,\underline{b}_{rs}$ satisfy $\overline{b}_{rs}\ne \underline{b}_{rs}$ ($r,\,s\in \Omega _n$).
Moreover, states $x_r(t),\,y_r(t)$ are assumed to be right continuous and have left limits. For convenience, several notations are given as follows: $\hat{b}_{rs}=\max \{\overline{b}_{rs},\,\underline{b}_{rs}\}$, $\check{b}_{rs}=\min \{\overline{b}_{rs},\,\underline{b}_{rs}\}$, $b^{**}_{rs}=\frac{1}{2}(\hat{b}_{rs}+\check{b}_{rs}) $, and $b^{*}_{rs}=\frac{1}{2}(\hat{b}_{rs}-\check{b}_{rs}) $. Then, systems (\ref{delay-impulse-memristive-drive-system}) and (\ref{delay-impulse-memristive-response-system}) are rewritten as
\begin{equation} \label{delay-impulse-memristive-drive-system-rewritten}
	\left\{
	\begin{array}{lcl}
		_{t_{\jmath -1}}\mathfrak{T} ^{\mathfrak{Q} }_{t} x_r(t)=-a_{r} x_r(t)+\sum_{s=1}^{n} (b^{**}_{rs}+\Psi    _{rs}(x_r(t)))f_s(x_s(t))+I_r,  & t\geq t_0,\,t\in [t_{\jmath -1},\,t_{\jmath }),\\
		x_r(t_{\jmath })=\mu _\jmath  x_r(t_{\jmath }^--\tau_\jmath ), & t= t_{\jmath },\,\jmath \in\mathbb{Z}^{+},\\
		x_r(t_0)=x_{r0}, & r\in \textcolor{black}{\Omega _n},\nonumber
	\end{array}
	\right.
\end{equation}
and
\begin{equation} \label{delay-impulse-memristive-response-system-rewritten}
	\left\{
	\begin{array}{lcl}
		_{t_{\jmath -1}}\mathfrak{T} ^{\mathfrak{Q} }_{t} y_r(t)=-a_{r} y_r(t)+\sum_{s=1}^{n} (b^{**}_{rs}+\Psi_{rs}(y_r(t)))f_s(y_s(t))+I_r
		+u_r(t),  & t\geq t_0,\,t\in [t_{\jmath -1},\,t_{\jmath }),\\
		y_r(t_{\jmath })=\mu _\jmath  y_r(t_{\jmath }^--\tau_\jmath ), & t= t_{\jmath },\,\jmath \in\mathbb{Z}^{+},\\
		y_r(t_0)=y_{r0}, & r\in \textcolor{black}{\Omega _n},\nonumber
	\end{array}
	\right.
\end{equation}
where
\begin{flalign*}
	\Psi    _{rs}(x_r(t))=\left\{
	\begin{aligned}
		b^{*}_{rs},    \qquad   & b_{rs} (x_r(t)) =  \hat{b}_{rs}\\
		-b^{*}_{rs}, \qquad  & b_{rs} (x_r(t)) =  \check{b}_{rs}\\
	\end{aligned}
	\right.
	,~\and~		\Psi    _{rs}(y_r(t))=\left\{
	\begin{aligned}
		b^{*}_{rs},    \qquad   & b_{rs} (y_r(t)) =  \hat{b}_{rs}\\
		-b^{*}_{rs}, \qquad  & b_{rs} (y_r(t)) =  \check{b}_{rs}\\
	\end{aligned}
	\right.	
	.
\end{flalign*}

Next, the Filippov theory \cite{20,21} is introduced to analyze the synchronization of delayed impulsive CFOMNNs (\ref{delay-impulse-memristive-drive-system}) and (\ref{delay-impulse-memristive-response-system}). \textcolor{black}{According to} the definition of set-valued map and the theory of differential inclusions given in \cite{20} and \cite{21}, systems (\ref{delay-impulse-memristive-drive-system}) and (\ref{delay-impulse-memristive-response-system}) can be transformed into
\begin{equation} \label{delay-impulse-memristive-drive-system-inclusion}
	\left\{
	\begin{array}{lcl}
		_{t_{\jmath -1}}\mathfrak{T} ^{\mathfrak{Q} }_{t} x_r(t)\in -a_{r} x_r(t)+\sum_{s=1}^{n} (b^{**}_{rs}+{co}[\Psi    _{rs}(x_r(t))])f_s(x_s(t))+I_r,  & t\geq t_0,\,t\in [t_{\jmath -1},\,t_{\jmath }),\\
		x_r(t_{\jmath })=\mu _\jmath  x_r(t_{\jmath }^--\tau_\jmath ), & t= t_{\jmath },\,\jmath \in\mathbb{Z}^{+},\\
		x_r(t_0)=x_{r0}, & r\in \textcolor{black}{\Omega _n},\nonumber
	\end{array}
	\right.
\end{equation}
and
\begin{equation} \label{delay-impulse-memristive-response-system-inclusion}
	\left\{
	\begin{array}{lcl}
		_{t_{\jmath -1}}\mathfrak{T} ^{\mathfrak{Q} }_{t} y_r(t)\in -a_{r} y_r(t)+\sum_{s=1}^{n} (b^{**}_{rs}+{co}[\Psi    _{rs}(y_r(t))])f_s(y_s(t))+I_r\\\qquad\qquad ~~~~
		+u_r(t),  & t\geq t_0,\,t\in [t_{\jmath -1},\,t_{\jmath }),\\
		y_r(t_{\jmath })=\mu _\jmath  y_r(t_{\jmath }^--\tau_\jmath ), & t= t_{\jmath },\,\jmath \in\mathbb{Z}^{+},\\
		y_r(t_0)=y_{r0}, & r\in \textcolor{black}{\Omega _n}.\nonumber
	\end{array}
	\right.
\end{equation}
From the measurable selection theorem in \cite{20} and \cite{21}, there exist functions $\gamma _{rs},\,\delta _{rs}\in {co}[-1,\,1]$ such that systems (\ref{delay-impulse-memristive-drive-system}) and (\ref{delay-impulse-memristive-response-system}) can be rewritten as
\begin{equation} \label{delay-impulse-memristive-drive-system-measurable}
	\left\{
	\begin{array}{lcl}
		_{t_{\jmath -1}}\mathfrak{T} ^{\mathfrak{Q} }_{t} x_r(t)= -a_{r} x_r(t)+\sum_{s=1}^{n} (b^{**}_{rs}+b^{*}_{rs}\gamma _{rs})f_s(x_s(t))+I_r,  & t\geq t_0,\,t\in [t_{\jmath -1},\,t_{\jmath }),\\
		x_r(t_{\jmath })=\mu _\jmath  x_r(t_{\jmath }^--\tau_\jmath ), & t= t_{\jmath },\,\jmath \in\mathbb{Z}^{+},\\
		x_r(t_0)=x_{r0}, & r\in \textcolor{black}{\Omega _n},
	\end{array}
	\right.
\end{equation}
and
\begin{equation} \label{delay-impulse-memristive-response-system-measurable}
	\left\{
	\begin{array}{lcl}
		_{t_{\jmath -1}}\mathfrak{T} ^{\mathfrak{Q} }_{t} y_r(t)= -a_{r} y_r(t)+\sum_{s=1}^{n} (b^{**}_{rs}+b^{*}_{rs}\delta _{rs})f_s(y_s(t)) +I_r+u_r(t),  & t\geq t_0,\,t\in [t_{\jmath -1},\,t_{\jmath }),\\
		y_r(t_{\jmath })=\mu _\jmath  y_r(t_{\jmath }^--\tau_\jmath ), & t= t_{\jmath },\,\jmath \in\mathbb{Z}^{+},\\
		y_r(t_0)=y_{r0}, & r\in \textcolor{black}{\Omega _n}.
	\end{array}
	\right.
\end{equation}

Let $e_r(t)=y_r(t)-x_r(t),\,r\in \Omega _n$ be the synchronization error. Then, it follows from (\ref{delay-impulse-memristive-drive-system-measurable}) and (\ref{delay-impulse-memristive-response-system-measurable}) that 
\begin{equation} \label{delay-impulse-memristive-error-system}
	\left\{
	\begin{array}{lcl}
		_{t_{\jmath -1}}\mathfrak{T} ^{\mathfrak{Q} }_{t} e_r(t)= -a_{r} e_r(t)+\sum_{s=1}^{n} (b^{**}_{rs}+b^{*}_{rs}\gamma _{rs})(f_s(y_s(t))-f_s(x_s(t)))
		\\\qquad\qquad ~~~~  +\sum_{s=1}^{n} b^{*}_{rs}(\delta _{rs}-\gamma _{rs})f_s(y_s(t))+u_r(t),  & t\geq t_0,\,t\in [t_{\jmath -1},\,t_{\jmath }),\\
		e_r(t_{\jmath })=\mu _\jmath  e_r(t_{\jmath }^--\tau_\jmath ), & t= t_{\jmath },\,\jmath \in\mathbb{Z}^{+},\\
		e_r(t_0)=y_{r0}-x_{r0}, & r\in \textcolor{black}{\Omega _n}.
	\end{array}
	\right.
\end{equation}
\begin{definition}\label{def-finite-time-synchronization}
	Systems (\ref{delay-impulse-memristive-drive-system}) and (\ref{delay-impulse-memristive-response-system})  are said to be finite-time synchronized, if there exists a scalar $T_0\in\mathbb{R}^+$, such that 
	\textcolor{black}{the error vector $e(t)=(e_1(t),\,\cdots,\,e_n(t))^{T}$} converges to zero in finite time, i.e.,
	\begin{equation}
		\lim_{t\rightarrow T_0}e(t)=0,\nonumber
	\end{equation}
	and
	\begin{equation}
		e(t)\equiv 0, \qquad \forall t\geq T_0,\nonumber
	\end{equation}
	where $T_0$ \textcolor{black}{is} the settling time.
\end{definition}

\begin{assumption}\label{assumption-memristive-lipschitz}
	For every $s\in \textcolor{black}{\Omega _n}$, there exist constants $L_{s},\,M_{s}\in\mathbb{R}^+$, such that
	\begin{equation}
		\left | f_s(z_1)-f_s(z_2) \right | \le L_s \left | z_1-z_2 \right |,\nonumber
	\end{equation}
and
\begin{equation}
	\left | f_s(z_1)\right | \le M_s,\qquad \forall z_1,\,z_2\in\mathbb{R} .\nonumber
\end{equation}
\end{assumption}

\textcolor{black}{We now} present a control \textcolor{black}{scheme} to \textcolor{black}{achieve} finite-time synchronization of drive-response delayed impulsive CFOMNNs (\ref{delay-impulse-memristive-drive-system}) and (\ref{delay-impulse-memristive-response-system}).
We design the controller $u_r(t),\,r\in \textcolor{black}{\Omega _n}$ as
\begin{equation}\label{delay-impulse-memristive-controller}
	u_r(t)=-\lambda_{r}e_r(t)-\zeta _r \sgn (e_r(t)) -\vartheta _r\left | e_r(t) \right |^{\varrho  } \sgn (e_r(t)),
\end{equation}
where $\lambda_{r}>0,\,\zeta _r>0,\,\vartheta _r>0$, and 
$\varrho  \in (0,\,1)$.
\begin{theorem}\label{theo-memristive-stable-impulses}
	Under Assumption \ref{assumption-memristive-lipschitz} and the controller (\ref{delay-impulse-memristive-controller}), delayed impulsive CFOMNNs (\ref{delay-impulse-memristive-drive-system}) and (\ref{delay-impulse-memristive-response-system}) can achieve finite-time synchronization, if the following conditions hold:
\begin{itemize}
	\item [$(H_1)$] $2(a_{r}+\lambda_{r})-\sum_{s=1}^{n}\frac{1}{2}\left [ \left(\left | \overline{b}_{rs}+\underline{b}_{rs} \right | +\left | \overline{b}_{rs}-\underline{b}_{rs} \right | \right) L_s  +\left(\left | \overline{b}_{sr}+\underline{b}_{sr} \right | +\left | \overline{b}_{sr}-\underline{b}_{sr} \right | \right) L_r \right ]\ge 0$;
	\item [$(H_2)$] $\zeta _r -\sum_{s=1}^{n} \left | \overline{b}_{rs}-\underline{b}_{rs} \right | M_s\ge 0$;
	\item [$(H_3)$] \textcolor{black}{$0<\hat{\beta } _\jmath =\mu _\jmath ^{ 1-\varrho  }\le \hat{\beta} <1$;}
	\item [$(H_4)$] There are $\hat{N}$ impulses on time interval $[0,\,\hat{\Gamma}_{e_0})$ and the $\hat{N}$-th impulsive \textcolor{black}{instant} satisfies 
	\begin{equation}\label{memristive-theo-stable-impulse-impulse-instant}
		t_{\hat{N}}\leq \bigg(\frac{\hat{\gamma}^{\hat{N}}(1-\frac{\hat{\beta}}{\hat{\gamma}})}{1-\hat{\beta}}\hat{\Gamma}_{e_0}^\mathfrak{Q} -\frac{\hat{\beta}}{1-\hat{\beta}}\tau^\mathfrak{Q} \bigg)^{\frac{1}{\mathfrak{Q} }}, \nonumber
	\end{equation}
	and
	\begin{equation}\label{memristive-theo-stable-impulse-gamma-condition}
		\hat{\gamma}^{\hat{N}}(1-\frac{\hat{\beta}}{\hat{\gamma}})\hat{\Gamma}_{e_0}^\mathfrak{Q} -\hat{\beta}\tau^\mathfrak{Q} >0,\nonumber
	\end{equation}
	where $\hat{\gamma}\in(\hat{\beta},\,1)$, $\hat{\Gamma}_{e_0}=\left(\frac{\left ( \frac{1}{2} \sum_{r=1}^{n} e_r^2(t_0) \right ) ^{1 -\hat{\eta}}}{\hat{\alpha}}\right)^{\frac{1}{\mathfrak{Q} }}$, \textcolor{black}{$\hat{\alpha}=\frac{\hat{c} (1  -\hat{\eta})}{ \mathfrak{Q} }>0$}, $\hat{c}=2^{\frac{1+\varrho  }{2}}\mathop{\min}_{1\le r\le n}  \left \{ {\vartheta _r } \right \}$, and $\hat{\eta}=\frac{1+\varrho  }{2}$.
\end{itemize}
Moreover, the settling time \textcolor{black}{is}  $\hat{T}_1=(\hat{\gamma}^{\hat{N}})^\frac{1}{\mathfrak{Q} }\hat{\Gamma}_{e_0}$.
\end{theorem}
\begin{proof}
	We define a piecewise right-continuous PDF function $\mathcal{V}( t ) \triangleq \mathcal{V}(e(t))$ for the error system (\ref{delay-impulse-memristive-error-system}) as
	\begin{equation}
		\mathcal{V}( t )=\frac{1}{2}e^T(t) e(t)=\frac{1}{2} \sum_{r=1}^{n} \left| e_r(t) \right|^2.\nonumber
	\end{equation}
When $t\in [t_{\jmath -1},\,t_{\jmath }),\,\jmath \in\mathbb{Z}^{+}$, by applying Lemma \ref{lemma-caputo-micidao-xiandaozaiji-absolute-daoshu}. (6), \textcolor{black}{we obtain} 
\begin{align}\label{proof-memristive-V-derivative}
	_{t_{\jmath -1}}\mathfrak{T} ^{\mathfrak{Q}  }_{t} \mathcal{V}( t )
	&=\textcolor{black}{\sum_{r=1}^{n}  e_r(t) \cdot {_{t_{\jmath -1}}\mathfrak{T} ^{\mathfrak{Q} }_{t}} e_r(t) } \nonumber\\
	&=  \sum_{r=1}^{n} \sgn(e_r(t))\left|e_r(t)\right| \cdot {_{t_{\jmath -1}}\mathfrak{T} ^{\mathfrak{Q} }_{t}} e_r(t)\nonumber\\
	&=\sum_{r=1}^{n} -(a_{r}+\lambda_{r})\left|e_r(t)\right|^2-\sum_{r=1}^{n} \vartheta _r \left | e_r(t) \right |^{1+\varrho  }  \nonumber \\
	&~ ~ ~ ~+\sum_{r=1}^{n}\sgn(e_r(t))\left|e_r(t)\right| \sum_{s=1}^{n} (b^{**}_{rs}+b^{*}_{rs}\gamma _{rs})(f_s(y_s(t))-f_s(x_s(t)))\nonumber \\
	&~ ~ ~ ~ +\sum_{r=1}^{n}\sgn(e_r(t))\left|e_r(t)\right|\sum_{s=1}^{n} b^{*}_{rs}(\delta _{rs}-\gamma _{rs})f_s(y_s(t))-\sum_{r=1}^{n} \zeta _r \left|e_r(t)\right|.
\end{align}
From Assumption \ref{assumption-memristive-lipschitz}, it yields that
\begin{align}\label{proof-memristive-V-derivative-f-Lipschitz}
	&~~~~ \sum_{r=1}^{n}\sgn(e_r(t))\left|e_r(t)\right| \sum_{s=1}^{n} (b^{**}_{rs}+b^{*}_{rs}\gamma _{rs})(f_s(y_s(t))-f_s(x_s(t)))\nonumber\\
	&\le \sum_{r=1}^{n}\sum_{s=1}^{n}\frac{1}{2}\left (\left | \overline{b}_{rs}+\underline{b}_{rs} \right | +\left | \overline{b}_{rs}-\underline{b}_{rs} \right | \right) L_s \left | e_r(t) \right |\cdot  \left | e_s(t) \right |\nonumber\\
	&\le \sum_{r=1}^{n}\sum_{s=1}^{n}\frac{1}{4}\left(\left | \overline{b}_{rs}+\underline{b}_{rs} \right | +\left | \overline{b}_{rs}-\underline{b}_{rs} \right | \right) L_s \left (| e_r(t) \right |^2+\left | e_s(t) \right |^2)\nonumber\\
	&= \sum_{r=1}^{n}\sum_{s=1}^{n}\frac{1}{4}\left \{ \left(\left | \overline{b}_{rs}+\underline{b}_{rs} \right | +\left | \overline{b}_{rs}-\underline{b}_{rs} \right | \right) L_s +\left(\left | \overline{b}_{sr}+\underline{b}_{sr} \right | +\left | \overline{b}_{sr}-\underline{b}_{sr} \right | \right) L_r \right \}  \left | e_r(t) \right |^2,
\end{align}
and
\begin{align}\label{proof-memristive-V-derivative-f-bound}
	&~~~~ \sum_{r=1}^{n}\sgn(e_r(t))\left|e_r(t)\right|\sum_{s=1}^{n} b^{*}_{rs}(\delta _{rs}-\gamma _{rs})f_s(y_s(t))-\sum_{r=1}^{n} \zeta _r \left|e_r(t)\right|\nonumber \\
	&\le \sum_{r=1}^{n}e_r(t)  \sum_{s=1}^{n} \left | \overline{b}_{rs}-\underline{b}_{rs} \right | M_s-\sum_{r=1}^{n} \zeta _r \left | e_r(t) \right | \nonumber \\
	&\le \sum_{r=1}^{n} \left \{ \sum_{s=1}^{n} \left | \overline{b}_{rs}-\underline{b}_{rs} \right | M_s- \zeta _r  \right \} \left | e_r(t) \right | .
\end{align}
Substituting (\ref{proof-memristive-V-derivative-f-Lipschitz}) and (\ref{proof-memristive-V-derivative-f-bound}) into (\ref{proof-memristive-V-derivative}), it follows from $\mathrm{(H_1)}$ and $\mathrm{(H_2)}$ that
\begin{align}\label{proof-memristive-V-derivative-fangsuo}
	_{t_{\jmath -1}}\mathfrak{T} ^{\mathfrak{Q}  }_{t} \mathcal{V}( t )
	&\le - \sum_{r=1}^{n} \Bigg \{ 2(a_{r}+\lambda_{r})-\sum_{s=1}^{n}\frac{1}{2}\left [ \left(\left | \overline{b}_{rs}+\underline{b}_{rs} \right | +\left | \overline{b}_{rs}-\underline{b}_{rs} \right | \right) L_s \right.\Bigg.\nonumber \\
	&\quad \Bigg.\left.  +\left(\left | \overline{b}_{sr}+\underline{b}_{sr} \right | +\left | \overline{b}_{sr}-\underline{b}_{sr} \right | \right) L_r \right ]  \Bigg \}  \frac{1}{2} \left | e_r(t) \right |^2 -\sum_{r=1}^{n} \left \{ \zeta _r -\sum_{s=1}^{n} \left | \overline{b}_{rs}-\underline{b}_{rs} \right | M_s \right \} \left | e_r(t) \right |\nonumber \\
	&~ ~ ~ ~ -\sum_{r=1}^{n} \vartheta _r \left | e_r(t) \right |^{1+\varrho  }\nonumber \\	
	&\le  -\sum_{r=1}^{n} \vartheta _r \left | e_r(t) \right |^{1+\varrho  }\nonumber \\	
	&\le  -2^{\frac{1+\varrho  }{2}}\mathop{\min}_{1\le r\le n}  \left \{ {\vartheta _r } \right \} \sum_{r=1}^{n} \left ( \frac{1}{2}\left | e_r(t) \right |^2 \right ) ^{\frac{1+\varrho  }{2} }\nonumber \\	
	&\le   -2^{\frac{1+\varrho  }{2}}\mathop{\min}_{1\le r\le n}  \left \{ \vartheta _r \right \}  \bigg ( \frac{1}{2} \sum_{r=1}^{n}  \left | e_r(t) \right |^2 \bigg ) ^{\frac{1+\varrho  }{2} }\nonumber \\	
	&= -\hat{c} \mathcal{V}^{\hat{\eta}}( t ),
\end{align}
where $\hat{c}=2^{\frac{1+\varrho  }{2}}\mathop{\min}_{1\le r\le n}  \left \{ {\vartheta _r } \right \}$ and $\hat{\eta}=\frac{1+\varrho  }{2}$. 

When $t= t_{\jmath },\,\jmath \in\mathbb{Z}^{+}$, one has
\begin{align}\label{proof-memristive-V-impulse-instant}
	\mathcal{V}( t_\jmath  )&=\frac{1}{2} \sum_{r=1}^{n} \left|e_r(t_\jmath )\right|^2\nonumber\\
	&=\frac{1}{2}\mu _\jmath ^2 \sum_{r=1}^{n}  \left|e_r(t_{\jmath }^--\tau_\jmath )\right|^2\nonumber\\
	&=\hat{\beta } _\jmath ^{\frac{1}{1 -\hat{\eta} } }\mathcal{V}( t_\jmath ^--\tau_\jmath  ),
\end{align}
where $\hat{\beta } _\jmath =\mu _\jmath ^{ 1-\varrho  }$.

Then, by using Theorem \ref{theo-finite-stable-impulse}, we conclude that systems (\ref{delay-impulse-memristive-drive-system}) and (\ref{delay-impulse-memristive-response-system}) are finite-time synchronized 
and the settling time is $\hat{T}_1=(\hat{\gamma}^{\hat{N}})^\frac{1}{\mathfrak{Q} }\hat{\Gamma}_{e_0}$, where $\hat{\gamma}\in(\hat{\beta},\,1)$, $\hat{\Gamma}_{e_0}=\left(\frac{\left ( \frac{1}{2} \sum_{r=1}^{n} e_r^2(t_0) \right ) ^{1 -\hat{\eta}}}{\hat{\alpha}}\right)^{\frac{1}{\mathfrak{Q} }}$, and $\hat{\alpha}=\frac{\hat{c} (1  -\hat{\eta})}{ \mathfrak{Q} }>0$.
\end{proof}

For impulsive CFOMNNs (\ref{delay-impulse-memristive-drive-system}) and (\ref{delay-impulse-memristive-response-system}), if $\tau_\jmath =0,\,\jmath \in \mathbb{Z}^+$, we can use Corollary \ref{coro-theo-stable-impulse} to obtain the following corollary directly. 

\begin{corollary}\label{coro-memristive-theo-stable-impulse}
	Consider \textcolor{black}{impulsive delays} $\tau_\jmath =0,\,\jmath \in \mathbb{Z}^+$. Under Assumption \ref{assumption-memristive-lipschitz} and the controller (\ref{delay-impulse-memristive-controller}), if the conditions ${(H_1)}-\, {(H_3)}$ in Theorem \ref{theo-memristive-stable-impulses} hold, then CFOMNNs (\ref{delay-impulse-memristive-drive-system}) and (\ref{delay-impulse-memristive-response-system}) are finite-time synchronized under delay-independent synchronizing impulses. Moreover, the settling time \textcolor{black}{is} $\check{T}_1=(\check{\gamma}^{\hat{N}})^\frac{1}{\mathfrak{Q} }\hat{\Gamma}_{e_0}$, if the impulsive sequence $\{t_{\jmath }\}_{\jmath \in\{1,\,\cdots,\,\hat{N}\}}$ satisfies
	 \begin{equation}\label{coro-memristive-theo-stable-impulse-impulse-instant}
			t_{\hat{N}}\leq \bigg(\frac{\check{\gamma}^{\hat{N}}(1-\frac{\hat{\beta}}{\check{\gamma}})}{1-\hat{\beta}}\bigg)^{\frac{1}{\mathfrak{Q} }}\hat{\Gamma}_{e_0}, \nonumber
		\end{equation}
	where $\check{\gamma}\in(\hat{\beta},\,1)$, $\hat{\Gamma}_{e_0}$
	is defined in Theorem \ref{theo-memristive-stable-impulses}.
\end{corollary}

Next, we study the case that CFOMNNs (\ref{delay-impulse-memristive-drive-system}) and (\ref{delay-impulse-memristive-response-system}) are subject to desynchronizing delayed impulses. 
\textcolor{black}{In light of} Theorem \ref{theo-finite-destable-impulse}, we \textcolor{black}{deduce} the finite-time synchronization result \textcolor{black}{as follows}. The \textcolor{black}{rigorous} proof is \textcolor{black}{similar} to the proof of Theorem \ref{theo-memristive-stable-impulses}, which is \textcolor{black}{omitted} here due to limited space.
\begin{theorem}\label{theo-memristive-destable-impulses}
	Under Assumption \ref{assumption-memristive-lipschitz} and the controller (\ref{delay-impulse-memristive-controller}), delayed impulsive CFOMNNs (\ref{delay-impulse-memristive-drive-system}) and (\ref{delay-impulse-memristive-response-system}) can achieve finite-time synchronization, if the conditions ${(H_1)}$ and ${(H_2)}$ in Theorem \ref{theo-memristive-stable-impulses} and
	\begin{itemize}
		\item [$(H_3^*)$] $0<\hat{\beta } _\jmath =\textcolor{black}{\mu _\jmath ^{ 1-\varrho  } } \le \hat{\beta}$, and $\hat{\beta}\ge 1$;
		\item [$(H_4^*)$] The impulsive sequence $\{t_{\jmath }\}_{\jmath \in\mathbb{Z}^{+}}$ satisfies
		\begin{equation}\label{memristive-theo-destable-impulse-impulse-instant}
			\hat{N}_0 \triangleq \min \{\jmath \in \mathbb{Z}^{+}:\,t_{\jmath }\ge (\hat{\gamma} ^{\jmath -1})^{\frac{1}{\mathfrak{Q} }}\hat{\Gamma} _{e_0} \}<+\infty ,\nonumber
		\end{equation}
		and
		\begin{equation}\label{memristive-theo-destable-impulse-gamma-condition}
			\hat{\gamma} \ge \hat{\beta} +\frac{\hat{\beta} \tau ^\mathfrak{Q} }{\hat{\Gamma} _{e_0}^\mathfrak{Q} } ,\nonumber
		\end{equation}
		where $\hat{\gamma}\in[\hat{\beta},\,+\infty )$, $\hat{\Gamma}_{e_0}$
		is defined in Theorem \ref{theo-memristive-stable-impulses}.
	\end{itemize}
	hold. Moreover, the settling time \textcolor{black}{is} $\hat{T}_2=(\hat{\gamma}^{\hat{N}_0-1})^\frac{1}{\mathfrak{Q} }\hat{\Gamma}_{e_0}$.
\end{theorem}

For impulsive CFOMNNs (\ref{delay-impulse-memristive-drive-system}) and (\ref{delay-impulse-memristive-response-system}), if impulses are delay-independent, i.e., $\tau_\jmath =0,\,\jmath \in \mathbb{Z}^+$, we can use Corollary \ref{coro-theo-destable-impulse} to obtain the following corollary 
directly. 

\begin{corollary}\label{coro-memristive-theo-destable-impulse}
	Consider \textcolor{black}{impulsive delays} $\tau_\jmath =0,\,\jmath \in \mathbb{Z}^+$. Under Assumption \ref{assumption-memristive-lipschitz} and the controller (\ref{delay-impulse-memristive-controller}), CFOMNNs (\ref{delay-impulse-memristive-drive-system}) and (\ref{delay-impulse-memristive-response-system}) are finite-time synchronized, if the conditions ${(H_1)}$ and ${(H_2)}$ in Theorem \ref{theo-memristive-stable-impulses}, ${(H_3^*)}$ in Theorem \ref{theo-memristive-destable-impulses} and
	\begin{itemize}
		\item [$(H_4^\Diamond)$] The impulsive sequence $\{t_{\jmath }\}_{\jmath \in\mathbb{Z}^{+}}$ satisfies
		\begin{equation}\label{coro-memristive-theo-destable-impulse-impulse-instant}
			\bar{N}_0 \triangleq \min \{\jmath \in \mathbb{Z}^{+}:\,t_{\jmath }\ge (\hat{\beta} ^{\jmath -1})^{\frac{1}{\mathfrak{Q} }}\hat{\Gamma} _{e_0} \}<+\infty ,\nonumber
		\end{equation}
	\end{itemize}
	hold. Moreover, the settling time \textcolor{black}{is} $\check{T}_2=(\hat{\beta}^{\bar{N}_0-1})^\frac{1}{\mathfrak{Q} }\hat{\Gamma}_{e_0}$, where $\hat{\Gamma}_{e_0}$
	is defined in Theorem \ref{theo-memristive-stable-impulses}.
\end{corollary}

\section{Examples}\label{sec: Numerical examples}

\begin{example}\label{example-general-system}
Consider the following CFODIS:
\begin{equation} \label{example-delay-impulse-system}
	\left\{
	\begin{array}{lcl}
		_{t_{\jmath -1}}\mathfrak{T} ^{0.98}_{t} \mathcal{S}
		(t)=-\frac{1}{3}   \mathcal{S}
		   ^\frac{1}{2}
		 , & t\geq t_0,\,t\in [t_{\jmath -1},\,t_{\jmath }),\\
		\mathcal{S} (t_{\jmath })=\alpha_\jmath  \mathcal{S}
		(t_{\jmath }^--\tau_\jmath ), & t= t_{\jmath },\,\jmath \in\mathbb{Z}^{+},\\
		\mathcal{S} (0)=0.5.
	\end{array}
	\right.
\end{equation}
We select the Lyapunov function $\mathcal{V}( t )=  \mathcal{S} ^2 (t)   $. Then, (\ref{theo-stable-impulse-V-yuanxitiong}) holds with $c=\frac{2}{3},\,\eta =\frac{3}{4}$. When there is no impulse, we can compute the settling time as $\Gamma_{\mathcal{S} _0}=4.280$ according to Lemma \ref{lemma-finite-time-no-impulse}. Figure \ref{fig-general-system-stable-impulse} (blue) and Figure \ref{fig-general-system-destable-impulse} (blue) show the trajectory and settling time of the system  (\ref{example-delay-impulse-system}) without impulses. \textcolor{black}{Next}, we shall consider the stabilizing and destabilizing delayed impulses, respectively.

$Case\,(i)$: (Stabilizing delayed impulses) We take $\alpha_\jmath =0.71$ and $\tau_\jmath   =0.05$ \textcolor{black}{for $\jmath \in\mathbb{Z}^{+}$}. Then, we have $\beta =0.843$. We choose $\gamma =0.9$. 
Suppose 
the impulsive instants are $\{0.2,\,0.4,\,4.4,\,\cdots\}$, which \textcolor{black}{meet} (\ref{theo-stable-impulse-impulse-instant}) and $N=2$. Hence, all conditions in Theorem \ref{theo-finite-stable-impulse} hold. The system (\ref{example-delay-impulse-system}) can achieve FTS with the settling time $T_1= 3.452$, which is shown in Figure \ref{fig-general-system-stable-impulse} (red).

Next, with the same parameters, we take the impulsive instants as $\{0.1,\,0.26,\,0.48,\,0.7,\,4.4,\\ \cdots\}$, which satisfy (\ref{theo-stable-impulse-impulse-instant}) and $N=4$. Hence, all conditions in Theorem \ref{theo-finite-stable-impulse} also hold. It can be observed from Figure \ref{fig-general-system-stable-impulse} (green) that the system (\ref{example-delay-impulse-system}) can achieve FTS with the settling time $T_1= 2.784$. 
Moreover, Figure \ref{fig-general-system-stable-impulse} also shows that the settling time decreases under the stabilizing delayed impulses.
\begin{figure}
	\begin{minipage}[t]{1.0\linewidth}
		\centering
		\includegraphics[width=3.5in]{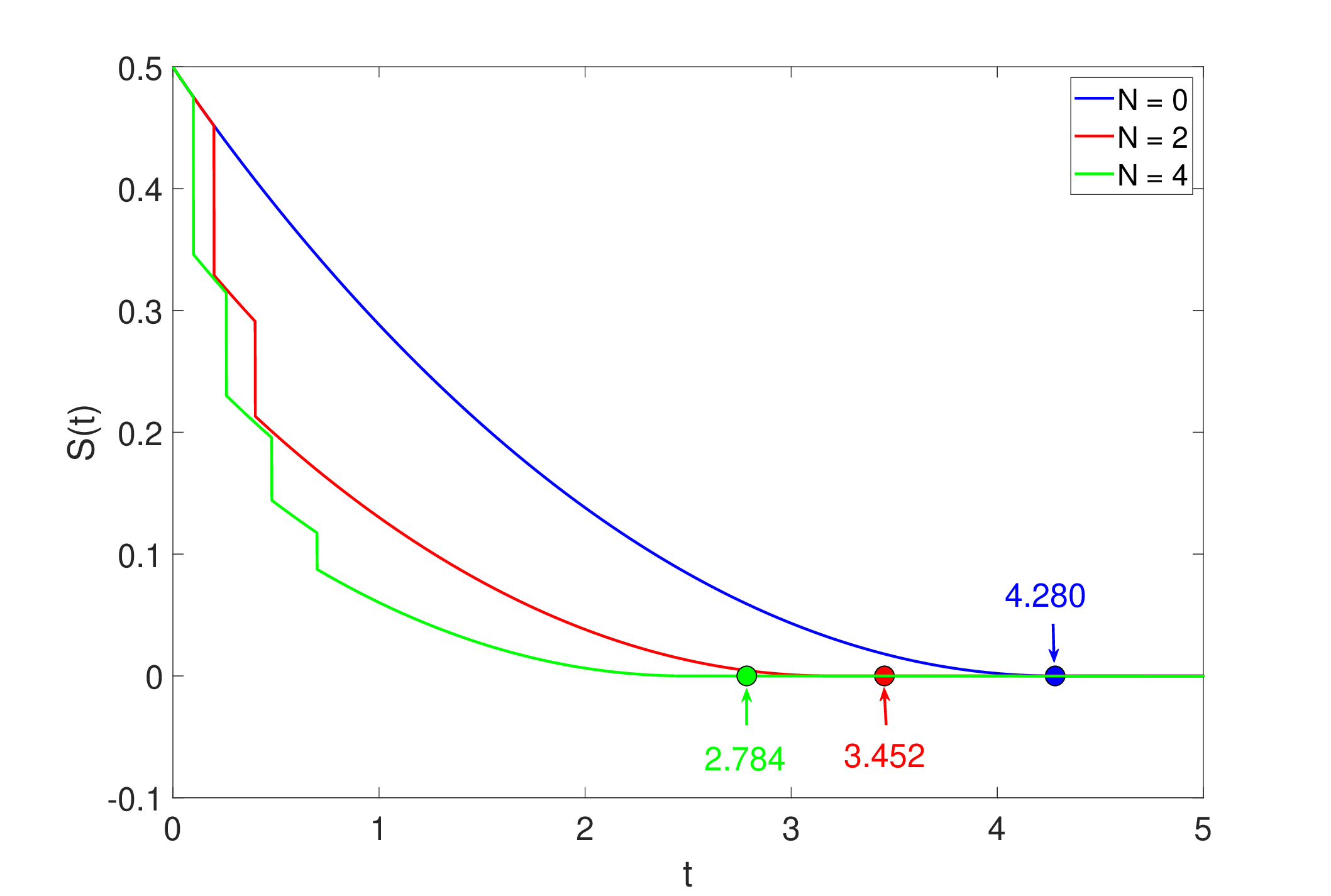}
	\end{minipage}
	\begin{center}
		\caption{Simulation results of $Case\,(i)$ in Example \ref{example-general-system}.}
		\label{fig-general-system-stable-impulse}
	\end{center}
\end{figure}

$Case\,(ii)$: (Destabilizing delayed impulses) We take $\alpha_\jmath =1.72$ and $\tau _\jmath =0.45$ \textcolor{black}{for $\jmath \in\mathbb{Z}^{+}$}. Then, we have $\beta =1.311$. We choose $\gamma =1.5$. Suppose the impulsive instants are $\{1.5,\, 2,\, 2.8,\, 15.5,\,\cdots\}$, which satisfy (\ref{theo-destable-impulse-impulse-instant}) and $N_0=4$. Hence, all conditions in Theorem \ref{theo-finite-destable-impulse} \textcolor{black}{hold}. It can be observed from Figure \ref{fig-general-system-destable-impulse} (red) that the system (\ref{example-delay-impulse-system}) can achieve FTS with the settling time $T_2= 14.810$. Moreover, compared with the blue curve in
Figure \ref{fig-general-system-destable-impulse}, we conclude that the settling time increases under the destabilizing delayed impulses.

Next, under the same parameters, we consider \textcolor{black}{impulsive delays} $\tau_\jmath  =0.1$. Hence, the conditions of Theorem \ref{theo-finite-destable-impulse} also \textcolor{black}{hold}. \textcolor{black}{We can see} from Figure \ref{fig-general-system-destable-impulse} (green) that the system (\ref{example-delay-impulse-system}) can achieve FTS.
Moreover, compared with the red curve in Figure \ref{fig-general-system-destable-impulse}, one observes that \textcolor{black}{impulsive delays have} a negative effect.
\begin{figure}
	\begin{minipage}[t]{1.0\linewidth}
		\centering
		\includegraphics[width=3.5in]{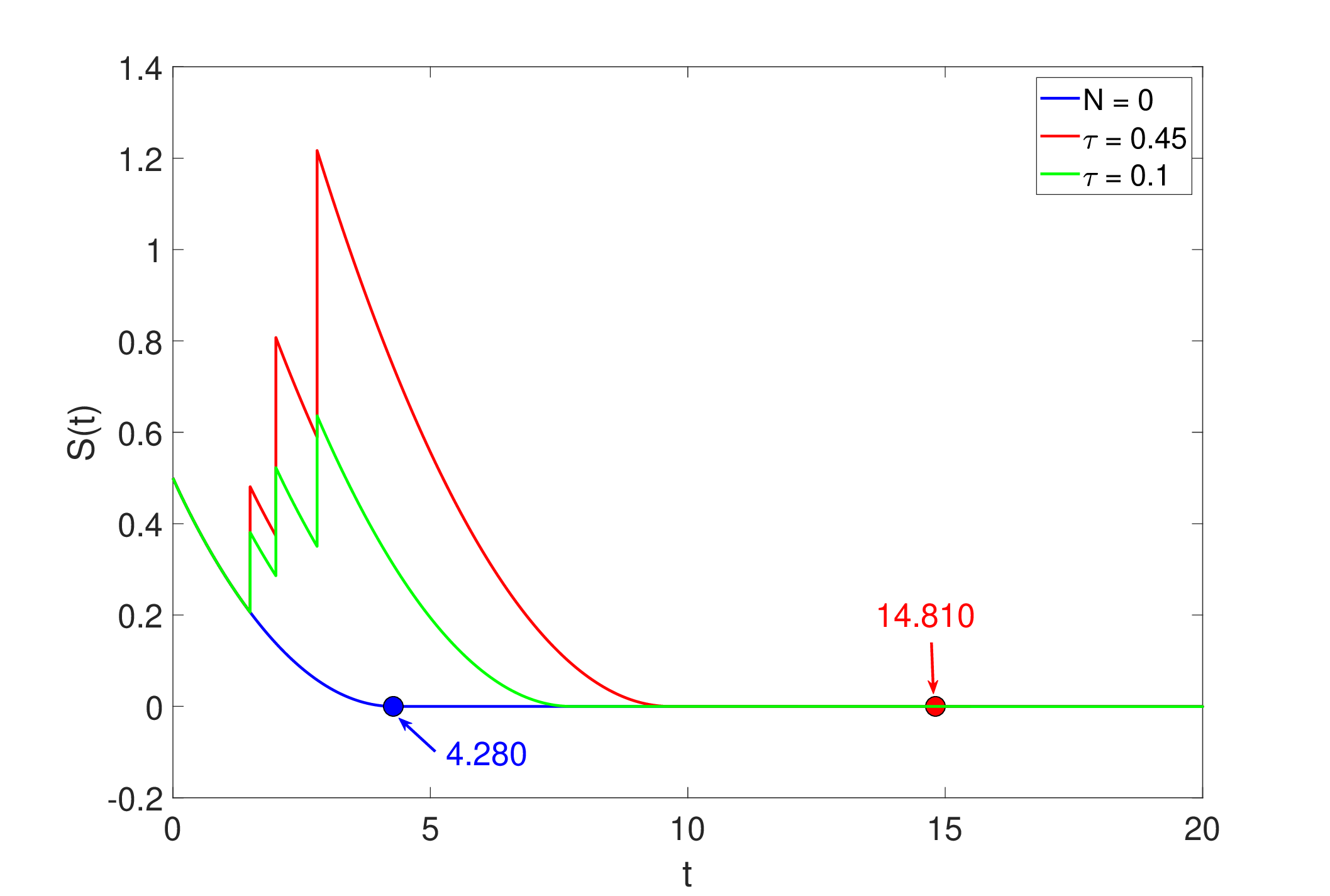}
	\end{minipage}
	\begin{center}
		\caption{Simulation results of $Case\,(ii)$ in Example \ref{example-general-system}.}
		\label{fig-general-system-destable-impulse}
	\end{center}
\end{figure}

\end{example}

\begin{example}\label{example-nn}
The parameters of drive-response CFOMNNs (\ref{delay-impulse-memristive-drive-system}) and (\ref{delay-impulse-memristive-response-system}) are selected as follows:
\begin{flalign*}
	b_{11} (x_1(t))=\left\{
	\begin{aligned}
		1.4,    \qquad   & \left | x_1(t) \right | \le  1\\
		1.5,    \qquad  & \left | x_1(t) \right | >  1\\
	\end{aligned}
	\right.
	, \qquad	b_{11} (y_1(t))=\left\{
	\begin{aligned}
		1.4,    \qquad   & \left | y_1(t) \right | \le  1\\
		1.5,    \qquad  & \left | y_1(t) \right | >  1\\
	\end{aligned}
	\right.,
\end{flalign*}
\begin{flalign*}
	b_{12} (x_1(t))=\left\{
	\begin{aligned}
		-1.3,    \qquad   & \left | x_1(t) \right | \le  1\\
		-1.2,    \qquad  & \left | x_1(t) \right | >  1\\
	\end{aligned}
	\right.
	, \qquad	b_{12} (y_1(t))=\left\{
	\begin{aligned}
		-1.3,    \qquad   & \left | y_1(t) \right | \le  1\\
		-1.2,    \qquad  & \left | y_1(t) \right | >  1\\
	\end{aligned}
	\right.,
\end{flalign*}
\begin{flalign*}
	b_{21} (x_2(t))=\left\{
	\begin{aligned}
		-2.1,    \qquad   & \left | x_2(t) \right | \le  1\\
		-2.6,    \qquad  & \left | x_2(t) \right | >  1\\
	\end{aligned}
	\right.
	, \qquad	b_{21} (y_2(t))=\left\{
	\begin{aligned}
		-2.1,    \qquad   & \left | y_2(t) \right | \le  1\\
		-2.6,    \qquad  & \left | y_2(t) \right | >  1\\
	\end{aligned}
	\right.,
\end{flalign*}
\begin{flalign*}
	b_{22} (x_2(t))=\left\{
	\begin{aligned}
		2.7,    \qquad   & \left | x_2(t) \right | \le  1\\
		2.3,    \qquad  & \left | x_2(t) \right | >  1\\
	\end{aligned}
	\right.
	, \qquad	b_{22} (y_2(t))=\left\{
	\begin{aligned}
		2.7,    \qquad   & \left | y_2(t) \right | \le  1\\
		2.3,    \qquad  & \left | y_2(t) \right | >  1\\
	\end{aligned}
	\right.,
\end{flalign*}
$\mathfrak{Q} = 0.93$, $\textcolor{black}{\Omega _2}=\{1,\,2\}$, $a_1=1.7,\,a_2=2.2$, and $I_1=I_2=0$. The initial values are set as $x_{1}(0)=2.5,\,x_2(0)=-3.9,\,y_{1}(0)=-4.7$, and $y_2(0)=9.8$. We take activation functions as $f_1(x_1(t))=1.3\tanh(x(t))$ and $f_2(x_2(t))=1.5\sin(x(t))$. Then, we have $L_1=1.3,\,M_1=1.3,\,L_2=1.5$, and $M_2=1.5$. We choose $\lambda_{1}=3.5,\,\lambda_{2}=4.9,\,\zeta _1=0.4,\,\zeta _2=1.5,\,\vartheta _1=1.1,\,\vartheta _2=1.2$, and $\varrho=0.3$. Hence, the conditions ${(H_1)}$ \textcolor{black}{and} ${(H_2)}$ in Theorem \ref{theo-memristive-stable-impulses} are satisfied. When there is no impulse, we can calculate the settling time as $\hat{\Gamma}_{e_0}=9.630$ according to Lemma \ref{lemma-finite-time-no-impulse}.
Figure \ref{fig-nn-stableimpulse_noimpulse} shows the error trajectories and settling time of drive-response CFOMNNs (\ref{delay-impulse-memristive-drive-system}) and (\ref{delay-impulse-memristive-response-system}) without impulses. \textcolor{black}{Next}, we shall consider the synchronizing and desynchronizing delayed impulses, respectively. 
\begin{figure}
	\begin{minipage}[t]{1.0\linewidth}
		\centering
		\includegraphics[width=3.5in]{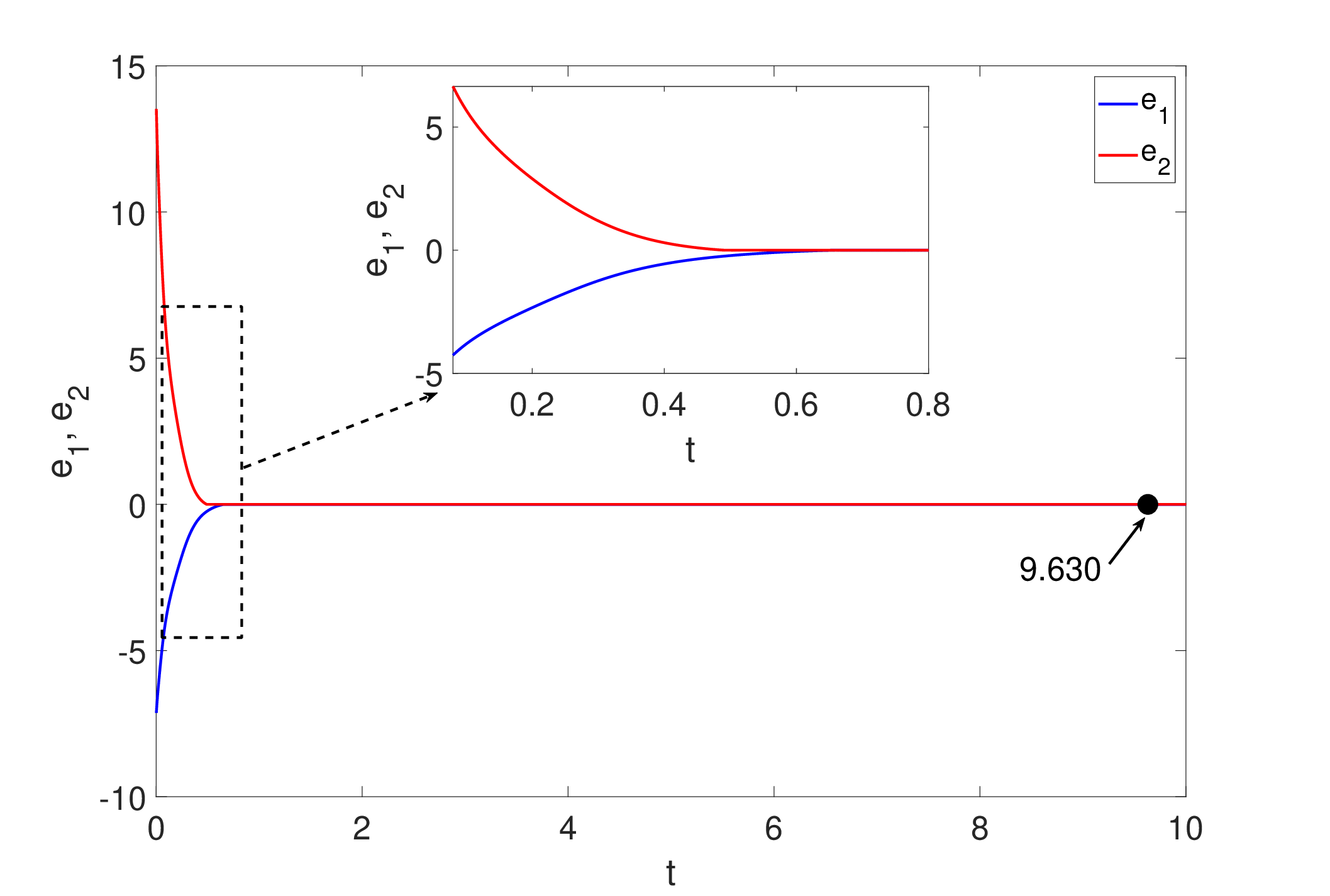}
	\end{minipage}
	\begin{center}
		\caption{Error trajectories of systems (\ref{delay-impulse-memristive-drive-system}) and (\ref{delay-impulse-memristive-response-system}) without impulses in Example \ref{example-nn}.}
		\label{fig-nn-stableimpulse_noimpulse}
	\end{center}
\end{figure}

$Case\,(i)$: (Synchronizing delayed impulses) We take $\mu _\jmath =0.4$ and $\tau_\jmath  =0.01$ for $\jmath \in\mathbb{Z}^{+}$. Then, we have $\hat{\beta} =0.527$. We choose $\hat{\gamma} =0.577$. Suppose the impulsive instants are $\{0.1,\,0.3,\,9.8,\,\cdots\}$, which satisfy ${(H_4)}$ and $\hat{N}=2$. 
Hence, all conditions in Theorem \ref{theo-memristive-stable-impulses} hold. The delayed impulsive CFOMNNs (\ref{delay-impulse-memristive-drive-system}) and (\ref{delay-impulse-memristive-response-system}) can achieve finite-time synchronization with the settling time $\hat{T}_1=2.946$, which is shown in Figure \ref{fig-nn-stableimpulse_delay0point02}.
\begin{figure}
	\begin{minipage}[t]{1.0\linewidth}
		\centering
		\includegraphics[width=3.5in]{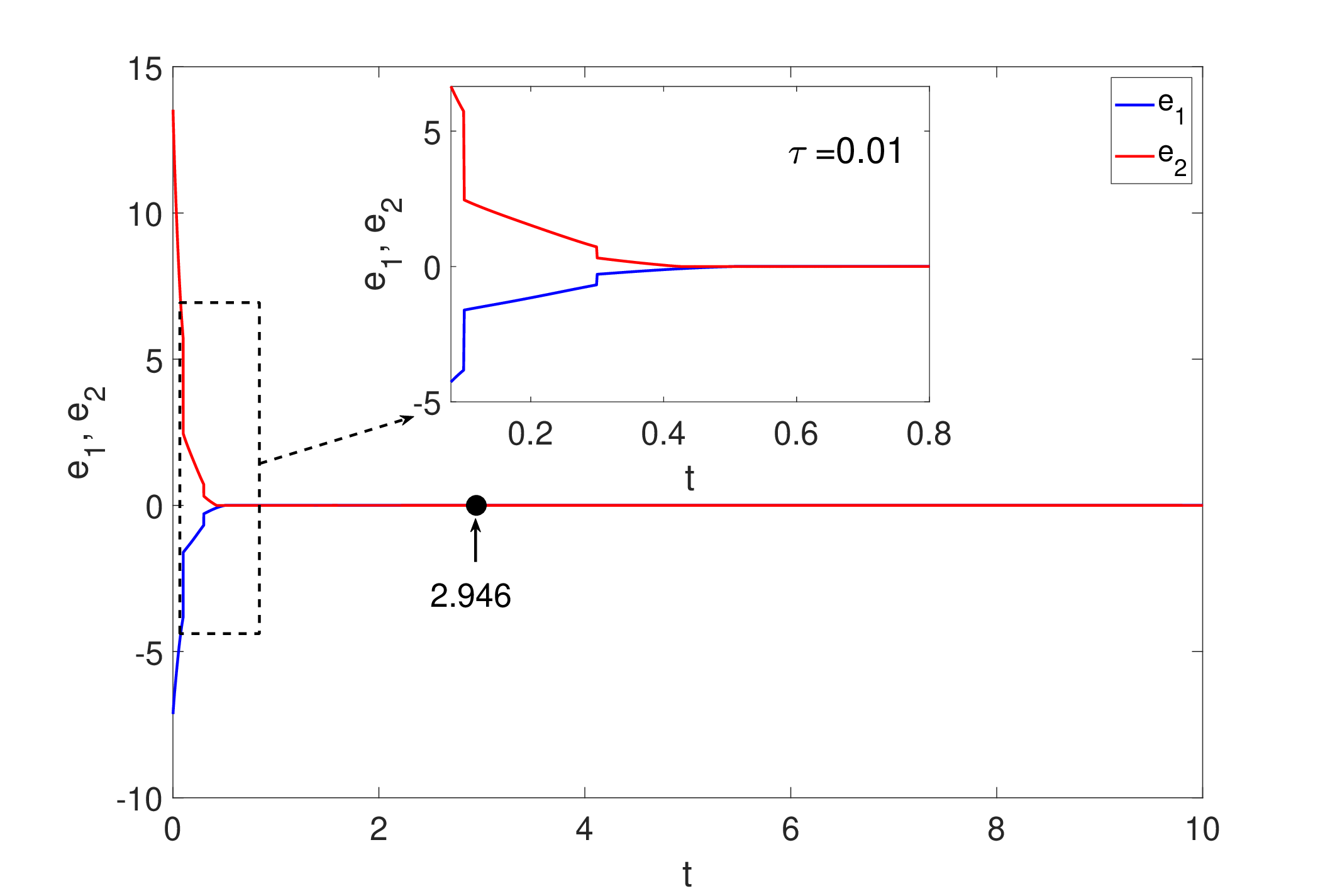}
	\end{minipage}
	\begin{center}
		\caption{Error trajectories of systems (\ref{delay-impulse-memristive-drive-system}) and (\ref{delay-impulse-memristive-response-system}) with $\tau=0.01$ in $Case\,(i)$ of Example \ref{example-nn}.}
		\label{fig-nn-stableimpulse_delay0point02}
	\end{center}
\end{figure}

Next, under the same parameters but $\tau_\jmath =0.08$ for $\jmath \in\mathbb{Z}^{+}$, the conditions in Theorem \ref{theo-memristive-stable-impulses} also hold. 
The error trajectories of systems (\ref{delay-impulse-memristive-drive-system}) and (\ref{delay-impulse-memristive-response-system}) are shown in Figure \ref{fig-nn-stableimpulse_delay0point08}. Compared with Figure \ref{fig-nn-stableimpulse_delay0point02}, \textcolor{black}{we can observe} that \textcolor{black}{impulsive delays have} a negative effect.

\begin{figure}
	\begin{minipage}[t]{1.0\linewidth}
		\centering
		\includegraphics[width=3.5in]{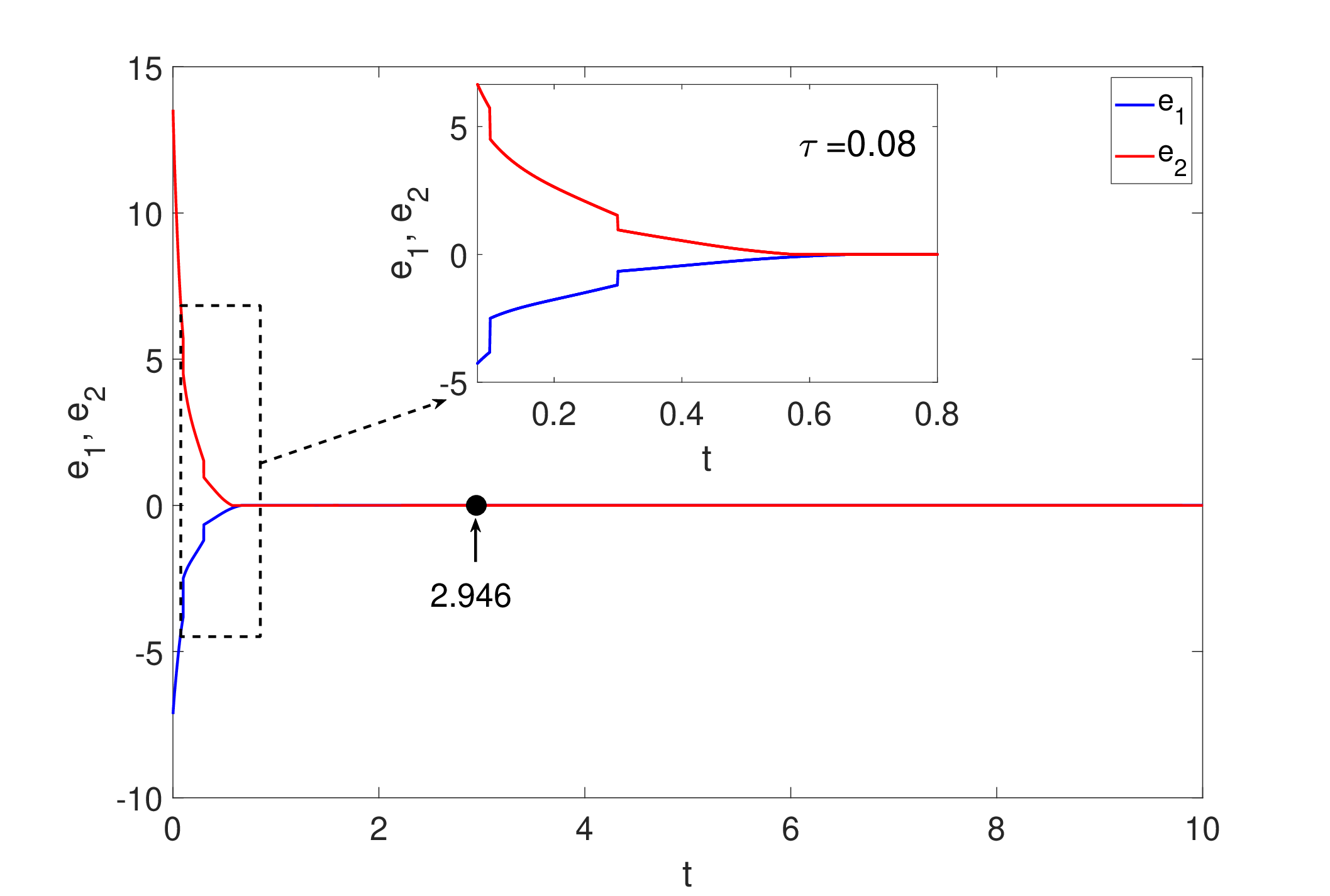}
	\end{minipage}
	\begin{center}
		\caption{Error trajectories of systems (\ref{delay-impulse-memristive-drive-system}) and (\ref{delay-impulse-memristive-response-system}) with $\tau=0.08$ in $Case\,(i)$ of Example \ref{example-nn}.}
		\label{fig-nn-stableimpulse_delay0point08}
	\end{center}
\end{figure}

$Case\,(ii)$: (Desynchronizing delayed impulses) Let $\mu _\jmath =1.38$ and $\tau_\jmath =0.005$ for $\jmath \in\mathbb{Z}^{+}$. Then, we have $\hat{\beta} =1.253$. We set $\hat{\gamma} =1.353$. Suppose the impulsive instants are $\{0.12,\,0.35,\,19,\,\\ \cdots\}$, which satisfy $(H_4^*)$ and $\hat{N}_0=3$. 
Hence, all conditions in Theorem \ref{theo-memristive-destable-impulses} \textcolor{black}{hold}. Delayed impulsive CFOMNNs (\ref{delay-impulse-memristive-drive-system}) and (\ref{delay-impulse-memristive-response-system}) can achieve finite-time synchronization with the settling time $\hat{T}_2=18.446$, which is shown in Figure \ref{fig-nn-destableimpulse_delay0point005}.
\begin{figure}
	\begin{minipage}[t]{1.0\linewidth}
		\centering
		\includegraphics[width=3.5in]{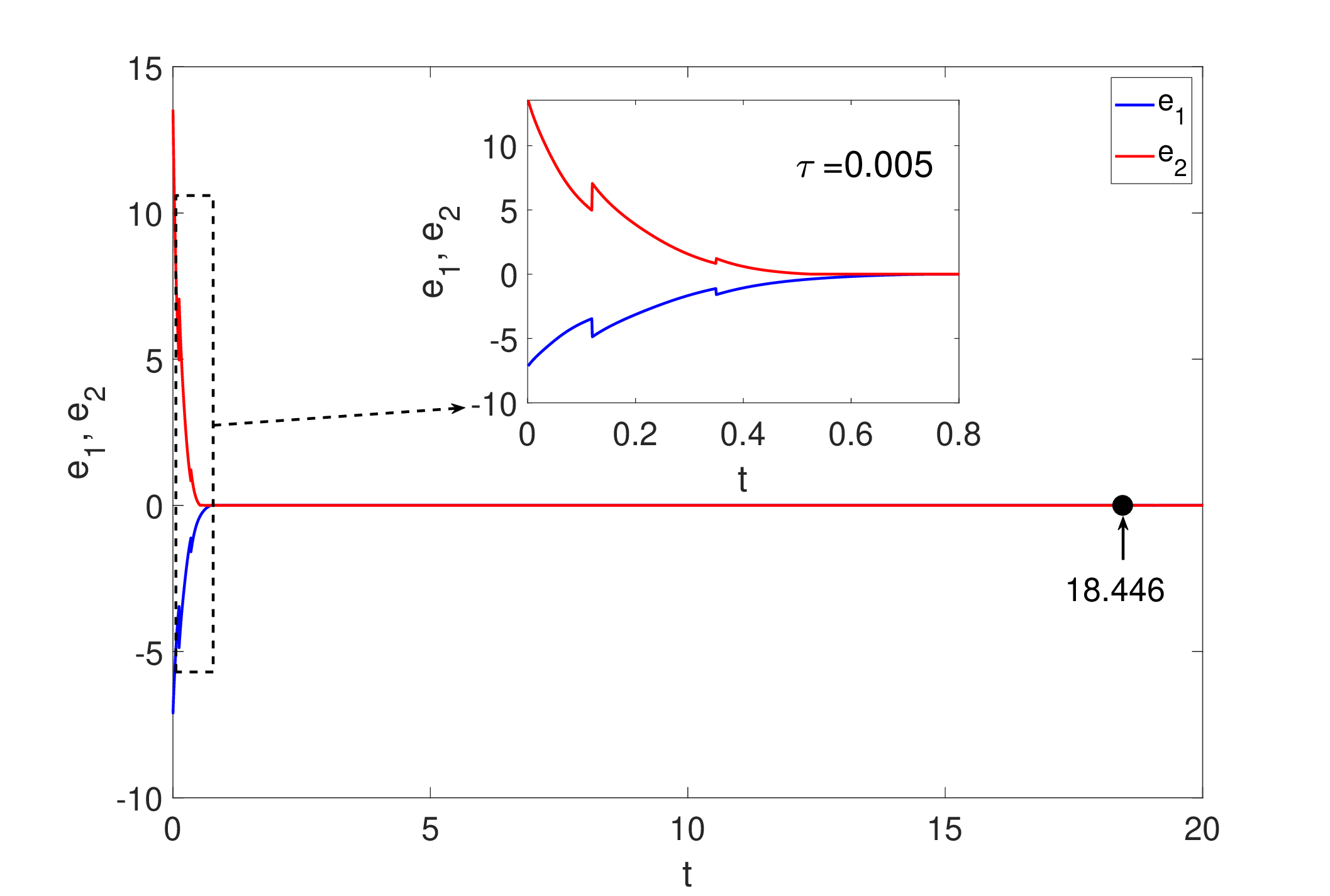}
	\end{minipage}
	\begin{center}
		\caption{Error trajectories of systems (\ref{delay-impulse-memristive-drive-system}) and (\ref{delay-impulse-memristive-response-system}) with $\tau=0.005$ in $Case\,(ii)$ of Example \ref{example-nn}.}
		\label{fig-nn-destableimpulse_delay0point005}
	\end{center}
\end{figure}

Next, under the same parameters but $\tau_\jmath =0.1$ for $\jmath \in\mathbb{Z}^{+}$, the conditions in Theorem \ref{theo-memristive-destable-impulses} also hold. 
The error trajectories of systems (\ref{delay-impulse-memristive-drive-system}) and (\ref{delay-impulse-memristive-response-system}) are shown in Figure \ref{fig-nn-destableimpulse_delay0point05}. Compared with Figure \ref{fig-nn-destableimpulse_delay0point005}, \textcolor{black}{we can observe} that \textcolor{black}{impulsive delays have} a negative effect.
\begin{figure}
	\begin{minipage}[t]{1.0\linewidth}
		\centering
		\includegraphics[width=3.5in]{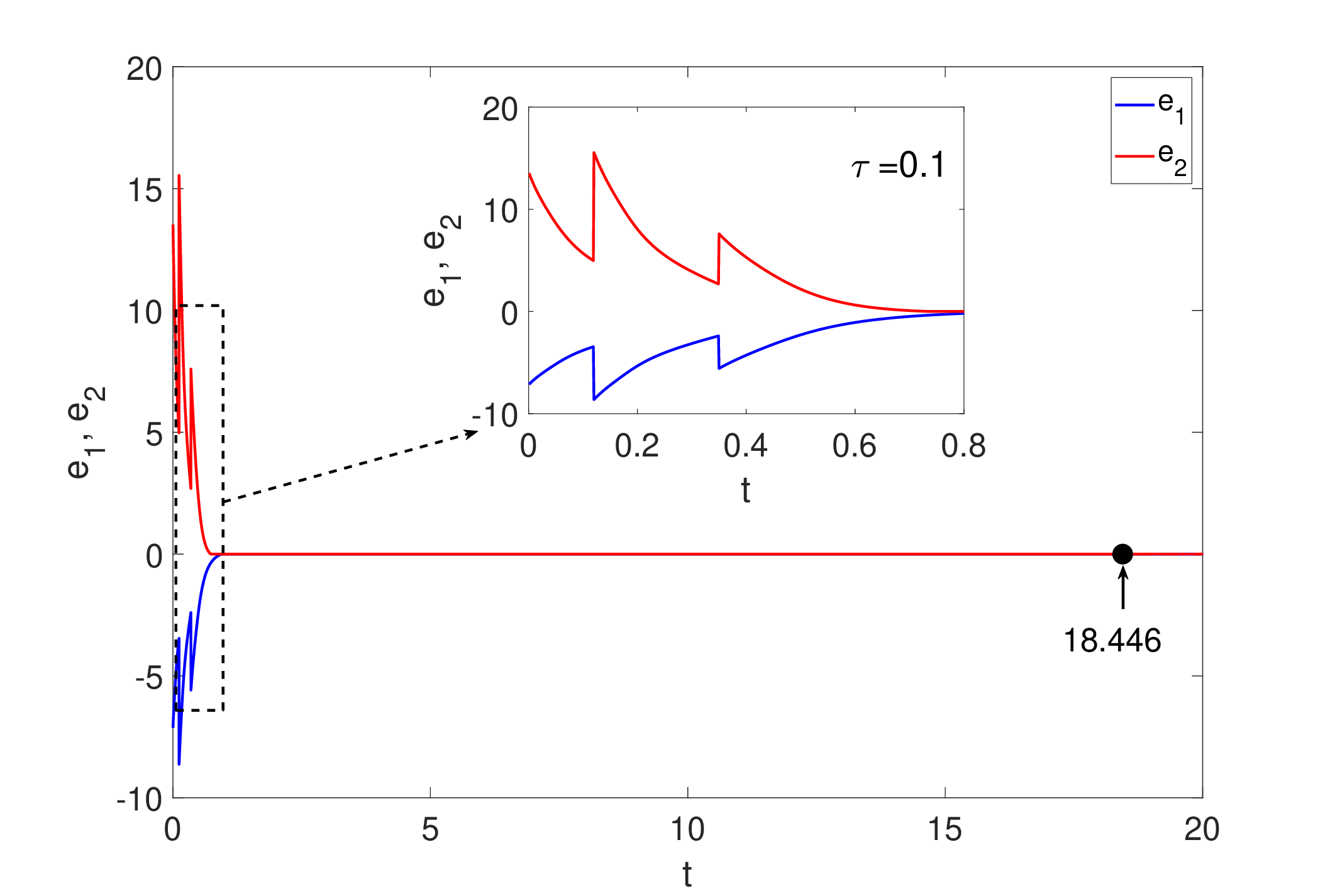}
	\end{minipage}
	\begin{center}
		\caption{Error trajectories of systems (\ref{delay-impulse-memristive-drive-system}) and (\ref{delay-impulse-memristive-response-system}) with $\tau=0.1$ in $Case\,(ii)$ of Example \ref{example-nn}.}
		\label{fig-nn-destableimpulse_delay0point05}
	\end{center}
\end{figure}

\end{example}

\section{Conclusion}\label{sec: Conclusion}
In this paper, the FTS of nonlinear CFODISs has been investigated. 
Some Lyapunov-based FTS theorems of CFODISs have been established when systems are subject to either stabilizing or destabilizing delayed impulses.
The settling time has been explicitly \textcolor{black}{estimated}, which depends on the initial values, the fractional order of systems and the impulses.
Then, some finite-time synchronization criteria of CFOMNNs with delayed impulses have been derived by applying the above theoretical results and designing a special controller. \textcolor{black}{Lastly}, simulations have been presented to \textcolor{black}{demonstrate} the \textcolor{black}{validity} of the theoretical \textcolor{black}{results}.


\section*{Declaration of Competing Interest}
The authors declare that they have no known competing financial interests or personal relationships that could have appeared to influence the work reported in this paper.

\section*{Acknowledgments}
This work was jointly supported by the National Natural Science Foundation of China (grant numbers 61503115 and 72271241) and the Natural Science Foundation of Anhui Province, China (grant number JZ2023AKZR0546).

\balance

\end{document}